\documentclass[12pt]{article}
\usepackage{amsmath,amssymb,amsthm,amscd}
\usepackage[all]{xy}
\usepackage{graphicx}

\newcommand{\Z}{\mathbb{Z}}
\newcommand{\Q}{\mathbb{Q}}
\newcommand{\R}{\mathbb{R}}
\newcommand{\C}{\mathbb{C}}
\newcommand{\T}{\mathbb{T}}

\newtheorem{defn}{Definition}[section]
\newtheorem{thm}[defn]{Theorem}

\newtheorem{prop}[defn]{Proposition}

\newtheorem{lemma}[defn]{Lemma}

\newtheorem{rem}[defn]{Remark}

\title{Some classifiable groupoid $C^{*}$-algebras with prescribed $K$-theory}

\author{
 Ian F. Putnam\thanks{Supported in part by a
grant from NSERC, Canada},\\
Department of Mathematics and Statistics,\\
University of Victoria,\\
Victoria, B.C., Canada V8W 3R4\\
email: ifputnam at uvic.ca}

\date{ }

\begin{document}
\maketitle

\begin{abstract}
Given a simple, acyclic dimension group $G_{0}$ 
and countable, torsion-free, abelian group $G_{1}$, we construct 
a minimal, amenable, \'{e}tale equivalence relation,
$R$, on a Cantor set whose associated groupoid 
$C^{*}$-algebra, $C^{*}(R)$,
is tracially AF, and hence 
classifiable in the Elliott classification scheme
for simple, amenable, separable $C^{*}$-algebras, 
and with $K_{*}(C^{*}(R)) \cong(G_{0}, G_{1})$.
\end{abstract}

\section{Introduction}
\label{intro}

Over the past twenty-five years, there has been an enormous amount
of truly
remarkable work devoted to classifying simple, amenable, 
separable  $C^{*}$-algebras by simple invariants
which can be roughly described as \newline K-theoretic.
This is the so-called Elliott program.

Let us discuss a specific important result in this program 
which is of interest to us here. In \cite{Lin1}, Huaxin Lin 
introduced the notion for a $C^{*}$-algebra to be tracially 
AF (or tracially approximately finite-dimensional).
We recall an equivalent form of the definition 
given by Dadarlat \cite{Dad}, in the case of a  unital 
$C^{*}$-algebra, $A$:
 for
any finite set $\mathcal{F}$ in $A$, $\epsilon > 0$ 
and non-zero projection $p_{0}$ in $A$, there 
is a finite-dimensional $C^{*}$-subalgebra
$F \subseteq  A$ with unit
$p$  satisfying
\begin{enumerate}
\item 
$\Vert pa - ap \Vert < \epsilon$, for all  $a$ in $\mathcal{F}$, 
\item 
$ pap \in_{\epsilon} F$, for all $a$ in $\mathcal{F}$, 
\item 
$1-p$ is unitarily equivalent to a subprojection of 
$p_{0}$.
\end{enumerate}

In \cite{Lin2}, Lin proved that two unital, separable, simple, amenable, 
tracially AF $C^{*}$ algebras satisfying the Universal
 Coefficient Theorem (see \cite{Bla})
are isomorphic if and only if
 their $K$-zero groups are isomorphic as ordered 
 abelian groups with order unit and their $K$-one groups are 
isomorphic.

One important aspect of the Elliott program 
is to understand the range of the invariant. In practical terms, this 
means specifying some $K$-theoretic data and constructing 
a $C^{*}$-algebra within the class having this data 
as its $K$-theory. The typical 
approach is to consider inductive limits of simpler
$C^{*}$-algebras. These are usually formed from 
continuous functions on compact spaces,
 finite-dimensional $C^{*}$-algebras, 
tensor products of these  and, 
 finally,  $C^{*}$-subalgebras of those.

This is quite natural in the sense that many of the approximation
techniques 
are very well-suited to dealing with these 
simpler algebras and also with inductive limits. 
On the other hand, many important $C^{*}$-algebras arise 
from geometric, topological, number theoretic
 or dynamical situations.  
In particular, the construction of $C^{*}$-algebras 
from groupoids provides a very general source for 
$C^{*}$-algebras \cite{Ren}. Moreover, this  has been the basis of
 many very fruitful interactions between these other fields 
 and operator algebras (\cite{GMPS2}, for example). The classification program 
 itself has devoted a lot of attention to the case of a crossed
 product of a discrete group acting on a commutative $C^{*}$-algebra
 (\cite{TW, Sz,ElZ}, as a few examples).
 In most of these cases, 
 inductive limit structures are not obviously available.

 In this paper, we address the following question: which 
 $C^{*}$-algebras that are classifiable in Elliott's sense, 
  may be constructed from 
 an \'{e}tale groupoid? Or, put in a better way, which 
  possible Elliott invariants may be realized as coming from
  the $C^{*}$-algebra of an \'{e}tale groupoid?
  If, in addition, 
   it is shown  that the $C^{*}$-algebra is 
  classifiable, then  the classification results 
  will provide isomorphisms with   $C^{*}$-algebras constructed by
  other means having the same invariant.

  The class of possible invariants we consider in our main 
  result is 
  rather restricted. On the other hand, the groupoids 
  which we produce are, in fact, equivalence relations 
  and the underlying space is a Cantor set. By a Cantor set, 
  we mean a compact, 
totally disconnected,  metrizable space with no isolated points.
 
 We recall the notion of an \'{e}tale equivalence relation.
 Let $X$ be a set and let $R$ be an equivalence relation on 
$X$. It is a groupoid
with the set of composable pairs $R^{2} =\{ ((x,y),(y,z)) \in R \times R \}$, 
product $(x,y) \cdot (y,z) = (x,z)$, for $((x,y),(y,z))$ in $R^{2}$ 
and inverse $(x,y)^{-1} = (y,x)$. For convenience, we will 
identify the diagonal in $R$ with $X$ in the usual way. This means that 
the range and source maps $r,s : R \rightarrow X$ are simply 
$r(x,y) = y,  s(x,y) = x$.

If $X,Y$ are topological spaces, a function 
$f: X \rightarrow Y$ is a local homeomorphism,
 if, for every $x$ in $X$, there is a neighbourhood $U$ of 
$x$ such that $f(U)$ is open and $f: U \rightarrow f(U)$ is 
a homeomorphism.

We say that a  topology
on an  equivalence relation $R$ on 
a topological space $X$ is \emph{\'{e}tale} 
 if the two maps $r,s$ are 
local homeomorphisms. We also say that 
$R$ is an \'{e}tale equivalence relation.
We  say that $R$ is minimal if every equivalence class is dense 
in $X$.

If $R$ is an \'{e}tale equivalence relation on a space $X$, the 
linear space of continuous complex functions of compact support 
on $R$ becomes a 
$*$-algebra with the operations
\begin{eqnarray*}
f \cdot g (x, y) &  =  & \sum_{(x,z) \in R} f(x,z) g(z,y), \\
f^{*}(x,y) & = & \overline{f(y,x)}
\end{eqnarray*}
for $f, g$ continuous and compact supported and $(x,y)$ in $R$.
The completion of this algebra in a suitable norm is 
a $C^{*}$-algebra which we denote by $C^{*}(R)$.

One particularly nice class of examples are the AF-equivalence relations. 
In this case, the associated $C^{*}$-algebra is an $AF$-algebra and 
its ordered $K$-zero group can be computed rather easily. It is a 
dimension group (see \cite{Eff} and \cite{Ren}).
 We say this ordered abelian group is simple if 
it has no order ideals. This is equivalent to the $C^{*}$-algebra
 being simple which, in turn, is equivalent to the 
 equivalence relation being minimal.
We describe this in detail in the next section.
 
There are a number of different equivalent conditions for a 
groupoid to be amenable. We will use 
 condition (ii) of Proposition 2.2.13 of \cite{ADR} which 
 follows.
There exists a sequence of functions $g_{l}, l \geq 1,$
on $R$ which are non-negative, continuous and compactly supported
and satisfy the following:
\begin{enumerate}
\item for every $x$ in $X$, 
\[
\sum_{(x,y) \in R } g_{l}(x,y) = 1,
\]
and 
\item the function
\[
(x,y) \rightarrow \sum_{(x,z) \in R}\vert g_{l}(x,z) - g_{l}(y,z) \vert,
\]
converges uniformly to zero on compact subsets of $R$.
\end{enumerate}
If $R$ is an amenable equivalence relation, then $C^{*}(R)$ is an amenable 
$C^{*}$-algebra \cite{ADR}. Moreover, in this case, $C^{*}(R)$ is simple if and only 
if $R$ is minimal \cite{Ren}.

  We state our main  result.

\begin{thm}
\label{intro:10}
Let $G_{0}$ be a simple, acyclic dimension group with order unit
 and let 
$G_{1}$ be a countable, torsion-free,  abelian group. There 
exists an \'{e}tale equivalence relation, $R$, on a Cantor set, $X$,
such that 
\begin{enumerate}
\item $R$ is minimal,
\item $R$ is amenable,
\item $K_{0}(C^{*}(R)) \cong G_{0}$, as ordered abelian groups with order unit,
\item $K_{1}(C^{*}(R)) \cong G_{1}$, as groups, and 
\item $C^{*}(R)$ is tracially AF.
\end{enumerate}
\end{thm}

In view of the classification
theorem of \cite{Lin2}, it is important to note  that Lemma 3.5 and Proposition
 10.7 of \cite{Tu}
show that the 
$C^{*}$-algebra of an \newline 
amenable, 
\'{e}tale equivalence relation satisfies the 
Universal Coefficient Theorem.
Hence, our $C^{*}(R)$ falls within the realm of Lin's result above and 
 the Elliott classification 
program generally.

Let us briefly discuss the main ingredients in the proof.
One begins with the dimension group $G_{0}$. This can be realized 
as the $K_{0}$-group of a simple AF-algebra and we begin with 
a Bratteli diagram, $(V, E)$, for it. This diagram has an infinite path space
which will be our Cantor set $X$. This space 
 has an \'{e}tale equivalence 
relation $R_{E}$ whose $C^{*}$-algebra is the AF-algebra.  Next, we 
find another Bratteli diagram, $(W, F)$,
 whose associated dimension group, 
without the order structure, is isomorphic to $G_{1}$ and we find two
disjoint embeddings of this diagram into $(V,E)$. The equivalence 
relation $R$ is then generated by the original $R_{E}$ and 
all equivalences between the two embedded copies of the paths 
of $(W, F)$. This description as an equivalence relation
is rather easy; what is more subtle is endowing $R$ 
with an \'{e}tale topology.

To verify the conditions of the theorem, the first, second and fifth 
parts are relatively straightforward. The main difficulty lies 
in the computation of the K-theory. The main tool
here is the results of \cite{Put2} and \cite{Put3}. We have $R_{E}$ as an open
subequivalence relation  of $R$ and hence we have an inclusion
$C^{*}(R_{E}) \subseteq C^{*}(R)$. The results of \cite{Put3} essentially allow
us to say that, since the difference between the groupoids 
$R_{E}$ and $R$ is described by the two embeddings 
of the diagram $(W,F)$, the relative K-theory for the inclusion
of their $C^{*}$-algebras can also be described from $(W,F)$.
This is then used to show that the inclusion of $C^{*}(R_{E})$ 
in $C^{*}(R)$ induces an isomorphism on $K$-zero groups
and that the  $K$-one group of $C^{*}(R)$ 
is $G_{1}$.

Let us make some comments on some special cases of the theorem
and other 
potential results along these lines.

Of course, the case $G_{1} = 0$ goes back to the seminal work 
of Elliott \cite{Ell1} on AF-algebras coupled with 
 Renault's construction of them from groupoids \cite{Ren}.
 The case $G_{1} \cong \Z$ is closely linked with the 
 work of the author with Giordano, Herman and Skau \cite{HPS}. 
 
The author, along with Deeley and Strung \cite{DPS}, showed 
that the Jiang-Su algebra  \cite{JS,RW}  could be realized via an
\'{e}tale equivalence relation.
This is a case when $G_{0} \cong \Z$,  $G_{1} = 0$, but 
the lack of projections in the $C^{*}$-algebra 
means that we cannot
use a Cantor set for the space $X$. 
The main idea  is to begin with a minimal, uniquely ergodic 
homeomorphism of a sphere of odd dimension at least $3$. 
There are a couple of ways in which this construction might be extended.
The first, already discussed in   \cite{DPS}, is to begin with a 
minimal homeomorphism of the sphere with a more complex set 
of invariant measures. This, of course, produces more traces 
on the $C^{*}$-algebras. If one  takes the product of 
one of these equivalence relations 
with one of the equivalence relations of Theorem \ref{intro:10},
one would obtain the $K$-theory from \ref{intro:10}, but with 
more traces. This produces examples where the 
$C^{*}$-algebras  are not
 real rank zero nor tracially AF.

Another extension of the results in 
\cite{DPS} is as follows. The key idea in \cite{DPS} is to alter the sphere
and the dynamics so as to insert 
a 'tube' which is invariant under the homeomorphism.
If one instead inserted $k$ of these tubes, one obtains a 
space $Z$ with a minimal homeomorphism $\zeta$ such that 
$K_{0}(C(Z) \times_{\zeta} \Z) \cong \Z^{k} \cong 
K_{1}(C(Z) \times_{\zeta} \Z) $.

Finally, it appears that the statement of \ref{intro:10}
 can be extended 
to include the case $G_{1}$ is finite. 
But this  work is still in progress
 and requires
 more general excision results than those 
 available in \cite{Put2} and \cite{Put3}.
  
The paper is organized as follows. In the second section, we provide 
background information on \'{e}tale equivalence relations. 
We also define our \'{e}tale equivalence relation $R$ and 
prove some basic properties, including amenability and minimality. 
A key part of the structure of $R$ is that it
 contains an open 
subequivalence relation, $R_{E}$, which is AF. 

The third section deals with the computation of the 
$K$-theory of $C^{*}(R)$. The key point here is the presence 
of the AF-equivalence relation $R_{E}$, which provides us 
with a $C^{*}$-subalgebra, $C^{*}(R_{E}) \subseteq C^{*}(R)$ 
and results from \cite{Put2} and \cite{Put3} which allow us to compute the 
relative $K$-theory of this pair. 

The fourth section is devoted to showing that $C^{*}(R)$ is tracially AF.

The author would like to thank Thierry Giordano, Christian Skau and Charles
Starling for helpful comments on an early version of the manuscript,
Marius Dadarlat for enlightening discussions and the referee for 
many helpful suggestions.

\section{Groupoids}
\label{groupoids}

We begin with a general disucssion of \'{e}tale equivalence relations, particularly those on a Cantor set,
 and AF-equivalence relations. 
 
Returning to our undergraduate days, we recall that a function
is defined as a set of ordered pairs. The first thing that 
one usually does with this definition is to forget it
and treat functions as a kind of black box. Here, however, it
is useful to keep this definition. The reason is simply
 that we are used to thinking of equivalence relations as 
 sets of ordered pairs  
 and our equivalence relations will be expressed as the union 
 of a collection of functions. This is particularly
 useful as these functions provide a basis for the \'{e}tale topology
 which we require. 
 
 By a partial homeomorphism between spaces $X$ and $Y$, we mean 
 a function that is a homeomorphism between some open subset of $X$ and 
 an open subset of $Y$.
 We will use following theorem, which is 
 designed specifically for use on the Cantor set. Note that 
 if $f, g \subseteq X \times X$ are partial  homeomorphisms, 
 then their composition is 
 \[
 f \circ g = \{ (x,y) \mid (x,z) \in f, (z,y) \in g \}.
 \]
For any clopen set $U$, we denote the identity function on $U$ 
by $id_{U} = \{ (x, x) \mid x \in U \}$.

\begin{thm}
\label{groupoids:5}
Let $X$ be a Cantor set and 
let $\Gamma$ be a collection of subsets of  $X \times X
$ satisfying the following:
\begin{enumerate}
\item \label{groupoids:5-2} each element $\gamma$ of $\Gamma$ is a partial 
 homeomorphism 
of $X$ with $r(\gamma)$ and $ s(\gamma)$ clopen,
\item \label{groupoids:5-4}$ \{ U \subseteq X \mid U \text{ clopen }, id_{U} \in \Gamma \}$ 
is a basis for the topology 
of $X$, 
\item \label{groupoids:5-6}for all $\gamma$ in $\Gamma$, $\gamma^{-1}$ is in $\Gamma$, 
\item \label{groupoids:5-8}if $\gamma_{1}, \gamma_{2}$ are in $\Gamma$, then so is 
$\gamma_{1} \circ  \gamma_{2}$,
\item \label{groupoids:5-10}if $\gamma_{1}, \gamma_{2}$ are in $\Gamma$, then so is 
$\gamma_{1} \cap  \gamma_{2}$.
\end{enumerate}
Let $R_{\Gamma}$ be the union of the elements of $\Gamma$. Then $\Gamma$ is a basis for a 
topology on $R_{\Gamma}$ in which it is an \'{e}tale equivalence relation.
\end{thm}

\begin{proof}
The fact that a collection of sets satisfying condition \ref{groupoids:5-10}
forms a base for a topology on its union is, of course, trivial.

Let $\gamma$ be in $\Gamma$ and denote by 
$r_{\gamma}$ the restriction of the range map $r$ to $\gamma$. 
As we assume that $\gamma$ is a partial homeomorphism
with clopen domain and range, $r_{\gamma}(\gamma)$ is clopen and 
$r_{\gamma}$ is a bijection from 
$\gamma$ to $r_{\gamma}(\gamma)$. We will show that if 
the former is given the relative topology on $R$ coming
from the base $\Gamma$ and the latter is given the relative
 topology of $X$, then $r_{\gamma}$ is a homeomorphism.
 
 Let $U$ be any clopen subset of $X$ such that 
 $id_{U}$ is in $\Gamma$. It follows
 that 
 \[
 r_{\gamma}^{-1}(U \cap r_{\gamma}(\gamma) ) = \gamma \circ id_{U},
 \]
 which is again in $\Gamma$ by  property \ref{groupoids:5-8}
  of $\Gamma$.
 It follows from property \ref{groupoids:5-4}
  that $r_{\gamma}$ is continuous.
 
 On the other hand, if $\gamma'$ is any other element of $\Gamma$, then 
 $\gamma \cap \gamma'$ is again in $\Gamma$ by \ref{groupoids:5-10}  and 
 hence must be a partial homeomorphism. It follows that 
 $r_{\gamma}(\gamma \cap \gamma') = r(\gamma \cap \gamma')$
 is clopen.  This proves that $r_{\gamma}^{-1}$ is also continuous. Hence,
 $r_{\gamma}$ is a homeomorphism.
 
 The rest of the proof is fairly routine and we omit the details.  
\end{proof}

The easiest example of such a collection is to begin with $\varphi$, a free 
action of a discrete group, $G$, on a Cantor set, $X$. That is, for every 
$g$ in $G$, $\varphi^{g}$ is a homeomorphism of 
$X$ and $\varphi^{g} \circ \varphi^{h} = \varphi^{gh}$, for all 
$g,h$ in $G$. The action is free if $\varphi^{g}(x) = x$ occurs only 
for $g=e$.
 Here,  the collection
$\varphi^{g}|_{U}$, where $g$ varies over $G$ and $U \subseteq X$ 
is clopen, satisfies the hypotheses of \ref{groupoids:5}
and the associated equivalence relation is the orbit relation.

The next class of examples are the AF-equivalence relations.
These may be described rather abstractly as those \'{e}tale 
equivalence relations on a Cantor set which can be written as 
the countable union of an increasing sequence of 
compact, open subequivalence relations \cite{GPS1}.
They can also be described more concretely as being constructed from a
Bratteli diagram.
By a Bratteli diagram, $(V, E)$, we mean a vertex set, $V$, which is 
the union of finite non-empty subsets $V_{n}, n \geq 0$, with 
$V_{0} = \{ v_{0} \}$ and an edge set, $E$, which is 
the union of finite non-empty subsets $E_{n}, n \geq 1$ along with 
initial and terminal maps 
$i: E_{n} \rightarrow V_{n-1}, t:E_{n} \rightarrow V_{n}$.
Associated to such a diagram, we have its infinite path space
\[
X_{E} = \{ (x_{1}, x_{2}, \ldots ) 
\mid x_{n} \in E_{n}, t(x_{n}) = i(x_{n+1}), n \geq 1 \}.
\]
For $m < n$, we let $E_{m,n}$ denote the finite paths from vertices in 
$V_{m}$ to vertices in $V_{n}$ with the obvious maps 
$i:E_{m,n} \rightarrow V_{m}$, 
$t:E_{m,n} \rightarrow V_{n}$. For each $p$ in $E_{m,n}$, we let 
\[
U(p)  = \{ x \in X_{E} \mid x_{i} = p_{i}, m < i \leq n \}.
\]
We observe that, for any such $p$, 
\[
U(p) = \cup_{i(e) = t(p)} U(pe),
\]
and that the sets on the right are pairwise disjoint.
It follows that sets $U(p), p \in E_{0,n}, n \geq 1$,
form the basis for a topology on $X_{E}$. In this topology, 
the sets $U(p)$ are closed as well as open 
and $X_{E}$ is totally disconnected. It is also metrizable.

 We now introduce an \'{e}tale equivalence relation on $X_{E}$.
It will be convenient to let 
$I_{n}$ denote the set of all pairs
 $(p,q)$ with $p,q$  in $E_{0,n}$ satisfying
  $t(p) = t(q)$, for $n \geq 1$.
  We  define 
 \[
 \beta_{p,q} = \{ (x,y) \in X_{E} \times X_{E}
  \mid x_{i} = p_{i}, y_{i} = q_{i}, 
 1 \leq i \leq n, x_{i} = y_{i}, i > n \},
 \]
 for $(p,q)$ in $I_{n}$.
 
 \begin{lemma}
 \label{groupoids:10}
Let  $(p,q), ( p', q')$ be in 
$I_{n}$,  $ n \geq 1$. We have 
\begin{enumerate} 
\item $\beta_{p,q}$ is a partial homeomorphism from
$U(p)$ to $U(q)$,
\item  $\beta_{p,p}$ is the identity function of $U(p)$,
\item $\beta_{p,q}^{-1} = \beta_{q,p}$,
\item $\beta_{p,q} \circ \beta_{p',q'}$ is empty unless
$q=p'$ and equals $\beta_{p,q'}$ in that case,
\item $\beta_{p,q} \cap \beta_{p',q'}$ is empty unless
$p=p'$ and $q=q'$.
\item 
\[
\beta_{p,q} = \cup_{i(e) = t(p)} \beta_{pe,qe},
\]
and the sets on the right are pairwise disjoint.
\end{enumerate}
\end{lemma}

The proof is trivial.

\begin{defn}
\label{groupoids:20}
Let $(V, E)$ be a Bratteli diagram. We let $\mathcal{B}_{E}$ denote the 
collection of all sets $\beta_{p,q}$, where $p,q$ are in 
$I_{n}$ with  $ n \geq 1$, along with the empty set
and 
we denote the 
associated \'{e}tale equivalence relation by $R_{E}$.
\end{defn}

It is an easy matter to see that the collection 
$\mathcal{B}_{E}$ satisfies the conditions of 
Theorem \ref{groupoids:5}. The third and fourth conditions 
follow from Lemma \ref{groupoids:10}, provided
one considers $\gamma_{1} = \beta_{p,q}$ and
 $ \gamma_{2} =\beta_{p',q'}$ where $p,q,p',q'$ are all 
 in the same $E_{0,n}$. To deal with the case 
 $(p,q)$ is in $I_{n}$ and $(p',q')$ is in $I_{m}$, 
 it suffices to observe that repeated application 
 of the last part of \ref{groupoids:10}
  allows us to write $\beta_{p,q}$ as 
  a union of $\beta_{p',q'}$ with 
 $p',q'$ in any $E_{0,m}$ with $m > n$.

For a Bratteli diagram as above, we let $b_{p,q}$ 
denote the characteristic function of the compact, open 
set $\beta_{p,q}$ for $(p,q)$ in $I_{n}$. This is a partial isometry
in $C^{*}(R_{E})$. Fixing a vertex $v$ in $V_{n}$, 
all $b_{p,q}$ with $t(p)=t(q)=v$, form 
 a system of matrix units. That is, 
if we let $k(v)$ be the number of paths in $E_{0,n}$ which terminate at $v$,
and index the entries of $M_{k(v)}(\C)$ by such paths, then the map
sending $b_{p,q}$ to the matrix which is $1$ in the $p,q$-entry and zero
elsewhere is an isomorphism between
\[
B_{v} = span\{ b_{p,q}  \mid p,q \in E_{0,n}, t(p) = t(q) = v \}
\]
and $M_{k(v)}$.
It is also easy to check that 
\[
B_{n} = span\{ b_{p,q}  \mid (p,q) \in I_{n}  \} = \oplus_{v \in V_{n}} B_{v}.
\]
Finally, one checks that $B_{n} \subseteq B_{n+1}$, for all $n$ and that 
the union of the $B_{n}$ is dense in $C^{*}(R_{E})$.

We next describe the construction of the equivalence relation
 $R$ whose existence is claimed in Theorem \ref{intro:10}.
We begin by choosing a Bratteli diagram, $(V, E)$, whose associated 
AF-equivalence relation $R_{E}$ on $X_{E}$ has \newline 
$K_{0}(C^{*}(R_{E})) \cong G_{0}$, as ordered abelian groups with order unit.
We note that $(V, E)$ is simple as $G_{0}$ is, in the sense that, 
after telescoping to a subsequence, we can assume that between every vertex
at some level $n$ and another at level $n+1$, there is at least one edge.
 Also, as $G_{0}$ is 
not a cyclic group, we may assume that every edge set in $E$ has at least 
two edges. Moreover, by telescoping and symbol splitting 
the diagram further,
 we can make the cardinality of
both vertex sets and edge sets grow arbitrarily.
(See \cite{GPS1} for more discussion on telescoping
of Bratteli diagrams.)

Next, as $G_{1}$ is a countable torsion-free abelian group, it has an 
embedding into $\R$. This can be seen as follows: first, the fact it is 
torsion free means that the natural inclusion into $G_{1} \otimes_{\Z} \Q$
is injective. This is followed by an inclusion of this countable, 
rational vector space into $\R$. 
With  the relative order, it becomes a simple dimension group
and we can find a Bratteli diagram for it, $(W, F)$. (The actual order
does not matter, only that we can write it as an inductive limit of free
abelian groups with maps between them given by positive matrices.)

Having chosen $(W, F)$, we adjust $(V, E)$ so
  that $\# V_{n} > \# W_{n}$, for all $n \geq 1$
and so that the number of edges between any two vertices in $V_{n-1}$ and $V_{n}$ 
exceeds  $2 \# F_{n}$. 
In consequence, we may choose graph
 embeddings $\xi: W \rightarrow V$ and 
 $\xi^{0}, \xi^{1}: F \rightarrow E$
  with  disjoint images. Observe that $\xi^{0}, \xi^{1}$ also define 
  continuous maps from $X_{F}$ into $X_{E}$.

We let $\beta_{p,q}, \mathcal{B}_{E}, R_{E}$ be as 
described above for the diagram
$(V, E)$. 

Recall that for any \'{e}tale equivalence relation $R$ on 
compact space $X$, a Borel measure $\mu$ on 
$X$ is $R$-invariant if, for any Borel set $U \subseteq R$ such that 
$r|_{U}, s|_{U}$ are injective, we have $\mu(r(U)) = \mu(s(U))$.

\begin{lemma}
\label{groupoids:25}
Let $\mu$ be any $R_{E}$-invariant probability measure 
on $X_{E}$. Then we have 
\[
\mu(\xi^{0}(X_{F})) = \mu(\xi^{1}(X_{F})) = 0.
\]
\end{lemma}

\begin{proof}
Let $ i = 0,1$. 
It is well-known (see \cite{Eff}) that a probability measure on 
an AF-equivalence relation $R_{E}$ is determined uniquely by a 
function $\mu: V \rightarrow [0, 1]$ satisfying 
$\mu(v_{0}) = 1$ and $\mu(v) = \sum_{i(e)=v} \mu(t(e))$, for 
all $v$ in $V$. 
From this and the hypothesis on the number of edges 
in $F_{n}$ that 
\[
\mu(\xi^{i}(F_{0,n})) \leq 2^{-n}, 
\]
for all $n \geq 1$. As $\xi^{i}(X_{F}) = \cap_{n} \xi^{i}(F_{0,n})$
and the measure is regular, this completes the proof.
\end{proof}

As before, we denote the set of all pairs $p,q$ in $E_{0,n}$
with $t(p) = t(q)$ by $I_{n}$ and now also denote the set of
all such pairs with 
$t(p) = t(q) \in \xi(W)$ by $I_{n}^{W}$.

Now suppose that $p, q$ are two paths in $I_{n}^{W}$, for
some $n \geq 1$.
We define  subsets 
$\lambda^{1,0}_{p,q},\lambda^{0,1}_{p,q}, \delta^{1,0}_{p,q}, \delta^{0,1}_{p,q}$ of 
$X_{E} \times X_{E}$  as follows.
First, we define 
$\lambda^{1,0}_{p,q} $ to be the set of all pairs:
\[ 
\begin{array}{c} 
( p_{1}, \ldots, p_{n}, \xi^{1}(f_{n+1}), \xi^{1}(f_{n+2}), \ldots ) \\
 ( q_{1}, \ldots, q_{n}, \xi^{0}(f_{n+1}), \xi^{0}(f_{n+2}), \ldots ) \\
f_{i} \in F_{i}, i > n \end{array}
\]
It is a simple matter to
 see that $\lambda^{1,0}_{p,q}$ is a bijection. In fact, 
 its inverse is given by using the same formula, simply interchanging 
 $p$ and $q$ and 
 $\xi^{0}$ and $\xi^{1}$. We denote $(\lambda^{1,0}_{p,q})^{-1}$ 
  by $\lambda^{0,1}_{q,p}$. Observe that the sets $\lambda_{p,q}^{i,1-i}$ 
  are all disjoint from $R_{E}$; that is no pair in one is tail equivalent.

These are not, however, partial homeomorphisms as their domains are 
closed but not open. We are going to extend the definition
to define $\delta^{1,0}_{p,q}$ which will be a partial homeomorhphism from
$\cup U(p\xi^{1}(f))$ to 
$\cup U(q\xi^{0}(f))$, where, in both, the union is over all $f$ in 
$F_{n+1}$ with $\xi(i(f)) = v$. We denote the former clopen set by 
$U^{1}(p)$ and the latter by $U^{0}(q)$.
Every point $x$ in 
$\cup_{f \in F_{n+1}} U(p\xi^{1}(f))$  begins with $p$, then 
contains at least one edge of $\xi^{1}(F)$. Then one of three things can occur:
all edges after $p$ are in $\xi^{1}(F)$ (as is the case for the 
the map $\lambda_{p,q}^{1,0}$ above), the first edge after $p$ not in 
$\xi^{1}(F)$ lies in $\xi^{0}(F)$ or the first edge after $p$ not in 
$\xi^{1}(F)$ is also not in $\xi^{0}(F)$. We treat these separately, as follows.

We define $\delta^{1,0}_{p,q}$ to be the union of $\lambda^{1,0}_{p,q}$ 
with all pairs of sequences of the   form:
\[
\begin{array}{c} (p_{1}, \ldots, p_{n}, \xi^{1}(f_{n+1}), \ldots, \xi^{1}(f_{k}), 
\xi^{0}(f_{k+1}), x_{k+2}, \ldots )  \\ 
(q_{1}, \ldots, q_{n}, \xi^{0}(f_{n+1}), \ldots, \xi^{0}(f_{k}), 
\xi^{1}(f_{k+1}), x_{k+2}, \ldots ) \\
k \geq  n+1, f_{n+1}, \ldots, f_{k+1} \in F,
\end{array}
\]
and also all pairs of the form:
\[
\begin{array}{c} (p_{1}, \ldots, p_{n}, \xi^{1}(f_{n+1}), \ldots, \xi^{1}(f_{k}), 
x_{k+1}, x_{k+2}, \ldots )  \\ 
 (q_{1}, \ldots, q_{n}, \xi^{0}(f_{n+1}), \ldots, \xi^{0}(f_{k}), 
x_{k+1}, x_{k+2}, \ldots ) \\
k > n+1, f_{n+1}, \ldots, f_{k} \in F, x_{k+1} \notin \xi^{0}(F) \cup \xi^{1}(F)
\end{array}
\]

We first observe that the three types of pairs are distinct. In fact, their 
images under the source map (and also under the  range map) 
are pairwise disjoint.
 It is also a simple matter to
 see that $\delta^{1,0}_{p,q}$ is a partial homeomorphism from $U^{1}(p)$ to
 $U^{0}(q)$, as claimed. 
 In fact, its inverse is given by using the same formula,
  simply interchanging $p$ and $q$ and 
 $\xi^{0}$ and $\xi^{1}$. We denote $(\delta^{1,0}_{p,q})^{-1}$ 
  by $\delta^{0,1}_{q,p}$. We can summarize this in the following way.
 
 \begin{lemma}
 \label{groupoids:27}
 Let $n \geq 1$, $(p,q)$ be in $I_{n}^{W}$ and $i = 0,1$. 
 As defined above, $\delta^{i, 1-i}_{p,q}$ is a partial homeomorphism 
 of $X_{E}$. Moreover, we have 
 \[
 \delta_{p,q}^{i,1-i} \delta_{q,p}^{1-i,i} = \cup_{ \{f \in F_{n+1}, \xi(i(f)) = t(p) \} }
  \beta_{p\xi^{i}(f), p\xi^{i}(f)} 
  \]
  where the sets in the union are pairwise disjoint.
 \end{lemma}

  The following is obvious from the definitions, but worth stating
  explicitly.
  
 \begin{lemma}
 \label{groupoids:30}
 For $p,q$  in $I_{n}^{W}$ and $i = 0,1$,
  we have 
 \[
 \delta^{i, 1-i}_{p,q} - R_{E} = \lambda^{i,1-i}_{p,q}.
 \]
 \end{lemma}
 
We need an analogue of property 5 of Lemma \ref{groupoids:10} for 
our elements $\delta^{i,1-i}_{p,q}$, which is the following.

\begin{lemma}
\label{groupoids:40}
Let $(p, q)$ be in $I_{n}^{W}$ with $t(p) = \xi(w)$, 
for some $w$ in $W_{n}$.
Define 
\begin{eqnarray*}
A & = &  \{f \in F_{n+1} \mid 
i(f) = w \} \\
 B &  =  &  \{ (f, f') \in  F_{n+1} \times F_{n+2} \mid \\ &  & 
 i(f) = w, t(f) = i(f')  \} \\
 C & = &  
 \{(f, e) \in F_{n+1} \times E_{n+2} 
  \mid  \\ &  &  
 i(f) = w, \xi(t(f)) = i(e), e_{k+1} 
 \notin \xi^{0}(F_{n+2}) \cup \xi^{1}(F_{n+2})  \}.
 \end{eqnarray*}

 Then we have 
 \begin{eqnarray*}
 \delta^{1,0}_{p,q} & = &  \left( \cup_{f \in A} \delta^{1,0}_{p \xi^{1}(f), q \xi^{0}(f)}
 \right) \\
 &   &    \cup \left( 
 \cup_{(f,f') \in B}
   \beta_{p \xi^{1}(f) \xi^{0}(f'), q \xi^{0}(f) \xi^{1}(f') }
 \right) \\
    &   & 
  \cup \left( \cup_{(f,e) \in C} \beta_{p \xi^{1}(f) e, q \xi^{0}(f) e}
 \right)
 \end{eqnarray*}
 Moreover, the three sets on the right are pairwise disjoint.
\end{lemma}

\begin{proof}
If $(x,y)$  is in  $\delta^{1,0}_{p,q}$, then 
$(x_{1}, \ldots, x_{n}) = p$, and $x_{n+1} = \xi^{1}(f_{n+1})$, for 
some $f_{n+1}$ in 
$F_{n+1}$. Now, there are three mutually distinct possibilities for $x_{n+2}$. 
The first is $x_{n+2}$ is in $\xi^{1}(F_{n+2})$. In this case,  $(x,y)$ 
 is in 
$\delta^{1,0}_{p \xi^{1}(f_{n+1}), q \xi^{0}(f_{n+1})}$. The second is that 
 $x_{n+2}$ is in $\xi^{0}(F_{n+2})$. In this case,  $(x,y)$ 
is in \newline
$\beta_{p \xi^{1}(f_{n+1}) \xi^{0}(f_{n+2}),
 q \xi^{0}(f_{n+1}) \xi^{1}(f_{n+2})}$. 
 Finally, if $x_{n+2}$ is  in neither $\xi^{0}(F)$ nor $\xi^{1}(F)$, 
then $(x,y)$ is in   
$\beta_{p \xi^{1}(f_{n+1}) x_{n+2},q \xi^{1}(f_{n+1}) x_{n+2}}$.
The reverse inclusions are clear.
\end{proof}

The first immediate consequence of the lemma is the following.

 \begin{thm}
 \label{groupoids:50}
 The collection
 \begin{eqnarray*}
 \mathcal{B} & =  &  \{ \beta_{p,q} \mid (p, q) \in I_{n}, n \geq 1 \}  \\
    &   &  
   \cup \{ \delta^{1,0}_{p,q}, \delta_{p,q}^{0,1} \mid (p, q) \in I_{n}^{W}, 
   n \geq 1 \} \\
    &  &  \cup \{ \emptyset \}
   \end{eqnarray*}
   satisfies the hypotheses of Theorem \ref{groupoids:5}.
   We let $R$ be the associated \'{e}tale equivalence relation.
   The equivalence relation $R_{E}$ is an open subequivalence relation of $R$.
 \end{thm}

\begin{proof}
The first three  conditions are clear. The fourth and fifth must be dealt with 
in several cases. If $\gamma_{1}$ and $\gamma_{2}$ are both of the form 
$\beta_{p,q}$, both properties  follow from Lemma \ref{groupoids:10}. 

The next case is to consider $\beta_{p,q}$ and $\delta^{1,0}_{p',q'}$, with 
$(p,q)$ in $I_{n}$ and $(p',q')$ in $I_{m}^{W}$. First assume that $n=m$.
In this case, it is a simple matter to check that 
\[
\beta_{p,q} \circ \delta^{1,0}_{p', q'} = \left\{ \begin{array}{cl}
 \delta^{1,0}_{p,q'},  & q = p' \\ \emptyset, & q \neq p' 
 \end{array} \right.
\] 
and 
\[
\delta^{1,0}_{p', q'} \circ \beta_{p,q}   = \left\{ \begin{array}{cl}
 \delta^{1,0}_{p',q},  & q' = p \\ \emptyset, & q' \neq p 
 \end{array} \right.
\] 
It is also simple to observe that for any $(x,y)$ in $\delta^{1,0}_{p',q'}$, 
$x_{n+1} \neq y_{n+1}$ and hence
$\beta_{p,q} \cap \delta^{1,0}_{p', q'} = \emptyset$.
Similar arguments apply if we replace $\delta^{1,0}_{p',q'}$ with 
$\delta^{0,1}_{q',p'}$.

Now let us suppose that $n < m$. Using repeated applications of 
part 5 of Lemma \ref{groupoids:10}, we 
can replace $\beta_{p,q}$ with a union of $\beta_{p'', q''}$ with 
$(p'', q'')$ in $I_{m}$ and then the previous case applies, for both properties.

Finally, let us suppose that $m < n$.  Here, repeated applications of 
Lemma \ref{groupoids:40} allow us to replace $\delta^{1,0}_{p',q'}$ with a union 
of $\delta^{1,0}_{p'',q''}$ with $(p'', q'')$ in $I_{n}^{W}$ and some 
$\beta_{p''',q'''}$ (it is not necessary to specify the lengths of $p'''$). 
Then the previous case and part 5 of \ref{groupoids:5} yields both of the  
desired properties.

The next case to consider is $\gamma_{1} = \delta^{1,0}_{p,q}$ 
and $\gamma_{2} = \delta^{1,0}_{p',q'}$. We again begin with $n=m$. Here, 
$r(\delta^{1,0}_{p,q}) = U^{0}(q)$ and $s(\delta^{1,0}_{p,q}) = U^{1}(p)$. 
It follows easily that $\delta^{1,0}_{p,q} \circ \delta^{1,0}_{p',q'} = \emptyset$
 while $\delta^{1,0}_{p,q} \cap \delta^{1,0}_{p',q'}$ is empty unless
 $p=p'$ and $q=q'$. 
 
 Next, we suppose that $n < m$. We use Lemma \ref{groupoids:10} to 
 replace $\delta^{1,0}_{p,q}$ with a union of 
 $\delta^{1,0}_{p'',q''}$ with $(p'', q'')$ in $I_{m}^{W}$ and 
 $\beta_{p''', q'''}$ with $(p''', q''')$ in $I_{m'}$, for varying $m'$. 
 Here both results follow from the previous case
  and the cases already established above.
 
 The final case to consider is $\gamma_{1} = \delta^{1,0}_{p,q}$ 
and $\gamma_{2} = \delta_{p',q'}^{0,1}$. We again begin with $n=m$. Here, 
$r(\delta^{1,0}_{p,q}) = U^{0}(q)$ and $s(\delta^{1,0}_{p,q}) = U^{1}(p)$. 
It follows easily that $\delta_{p,q} \cap \delta_{p',q'}^{-1} = \emptyset$
 while $\delta_{p,q}^{1,0} \circ \delta_{p',q'}^{0,1}$ is empty unless
 $p=p'$ and $q=q'$ and then the composition is simply 
 the identity function on $U^{1}(p)$.
 
 Next, we suppose that $n < m$. We use Lemma \ref{groupoids:10} to 
 replace $\delta^{1,0}_{p,q}$ with a union of 
 $\delta^{1,0}_{p'',q''}$ with $(p'', q'')$ in $I_{m}^{W}$ and 
 $\beta_{p''', q'''}$ with $(p''', q''')$ in $I_{m'}$, for varying $m'$. 
 Here, both results follow from the previous case
  and the cases already established above.
\end{proof}

   The first desired property of our 
   equivalence relation $R$ is easy: as $R_{E}$ is a subequivalence 
   relation and is minimal,  every  $R_{E}$-equivalence class is dense in
   $X$ and hence the same is true for $R$.
   
   \begin{thm}
   \label{groupoids:55}
   The equivalence relation $R$ is minimal.
   \end{thm}
   
   Although we do not need it immediately,
   we record the following useful result.
   
\begin{prop}
\label{groupoids:60}
A probability measure $\mu$ on $X_{E}$ is 
 $R$-invariant if and only if it is  $R_{E}$-invariant. 
\end{prop}

\begin{proof}
The 'only if' part is clear. Now suppose that $\mu$ is $R_{E}$-invariant.
If we let $Y$ be the set of all points which are $R_{E}$-equivalent to 
some point of $\xi^{0}(X_{F}) \cup \xi^{1}(X_{F})$, this Borel set 
has $\mu(Y)=0$. On the other hand, $R_{E}$ and $R$ agree on the 
complement of $Y$ and it follows that $\mu$ is $R$-invariant as well.
\end{proof}

To study $R$ in greater detail,  for each $n \geq 1$, we let $\mathcal{B}_{n}$
be the collection of all sets 
 $\beta_{p,q} \subseteq R$, with $(p, q)$ in $I_{n}$, along with all
$\delta^{1,0}_{p,q}$ and $\delta_{p,q}^{0,1}$, with $(p, q)$ in $I_{n}^{W}$.
We also define $R_{n}$ 
to be the union 
of these sets.

\begin{lemma}
\label{groupoids:70}
\begin{enumerate}
\item
For each $n \geq 1$, $R_{n}$, with the relative topology of $R$,
 is a compact, open, \'{e}tale equivalence relation.
 We let $[x]_{R_{n}}$ be the equivalence class of $x$ 
 in $R_{n}$, for any $x$ in $X_{E}$.
 \item 
 For any $x$ in $X_{E}$, we 
 have 
 \[
 \#  [x]_{n} = \left\{ \begin{array}{cl} 2 \# \{ p \in E_{0,n} \mid t(p) = t(x_{n}) \},  &  x_{n+1}  \in \xi^{0}(F_{n+1}) \cup \xi^{1}(F_{n+1}) \\
  \# \{ p \in E_{0,n} \mid t(p) = t(x_{n}) \},  &  x_{n+1}  \notin \xi^{0}(F_{n+1}) \cup \xi^{1}(F_{n+1}) \end{array} \right.
  \]
\item 
For $l > n$, we have  
$R_{n} - R_{l} $ is the union of 
\[
\beta_{p \xi^{1}(f) \xi^{0}(f'), q \xi^{0}(f) \xi^{1}(f') } \cup 
 \beta_{ q \xi^{0}(f) \xi^{1}(f'), p \xi^{1}(f) \xi^{0}(f'), } 
\]
over all $(p,q)$ in $I_{n}^{W}$ and 
$ (f, f') $ in $F_{n,l} \times F_{l+1}$ with $ t(f) = i(f') $.
\item 
For fixed $n \geq 1 $, the sets 
$R_{n} - R_{l}, l > n,$ are pairwise disjoint.
\item 
For all $n \geq 1$, $R_{n} \subseteq R_{n+1} \cup R_{n+2}$. 
\item 
The union of $R_{n}, n \geq 1$, is $R$.
\end{enumerate}
\end{lemma}

\begin{proof}
Each of the sets in $\mathcal{B}$ is compact in the topology
of $X \times X$ and hence also in $R$. They are also
open in $R$ since they are part of the basis for its topology.
 Thus, $R_{n}$ is the union of a finite number of compact open sets 
 and is then compact and open.
 
 Let $x$ be any point in 
 $X$. If $x_{n+1}$ is not in $\xi^{0}(F_{n+1}) \cup \xi^{1}(F_{n+1})$, then 
 its images under the elements of $\mathcal{B}_{n}$ is exactly the set
 of all $y$ such that $y_{i} = x_{i}$, for all $i > n$.
  If $x_{n+1} = \xi^{0}(f_{n+1}) $, for some $f_{n+1}$ in $F_{n+1}$,
   then 
 its images under the elements of $\mathcal{B}_{n}$  is exactly the set
 of all $y$ such that $y_{i} = x_{i}$, for all $i > n$, and all $z$ with 
 $z_{n+1} = \xi^{1}(f_{n+1})$ and 
 $z_{i} = x_{i}$, for all $i > n+1$. 
 Similarly,  if $x_{n+1} = \xi^{1}(f_{n+1}) $, for some $f_{n+1}$ in $F_{n+1}$,
   then 
 its images under the elements of $\mathcal{B}_{n}$  is exactly the set
 of all $y$ such that $y_{i} = x_{i}$, for all $i > n$, and all $z$ with 
 $z_{n+1} = \xi^{0}(f_{n+1})$ and 
 $z_{i} = x_{i}$, for all $i > n+1$. 
 From this, it is clear that $R_{n}$ is an equivalence relation.

The second part follows immediately from the 
description of the equivalence classes we have just provided.

For the third part, we consider an element of $\mathcal{B}_{n}$. 
If it is $\beta_{p,q}$ for some $(p,q)$ in $I_{n}$, then by 
part 5 of Lemma \ref{groupoids:10}, it is contained in $R_{l}$
and so $\beta_{p,q} - R_{l} $ is empty.
 For $\delta^{1,0}_{p,q}$ with $(p,q)$ in $I_{n}^{W}$, we apply 
 Lemma \ref{groupoids:40} $l-n$ times. In the end, we have 
 a collection of $\delta^{1,0}_{p',q'}$ where $(p',q')$ is in 
 $I_{l}^{W}$, which are all in $R_{l}$. We also have many 
 $\beta_{p',q'}$. Almost all of these are contained in 
 $R_{l}$, with  the only  exception being those listed in the 
 statement where the $(p',q')$ are in $I_{l+1}^{W}$. 
 The reverse containment is clear.

 From the description of  $R_{n} - R_{l}$ which is given 
 in the third part, we see that any pair $(x,y)$ in $R_{n} - R_{l}$ 
 has $x_{l+1} \neq y_{l+1}$, but $x_{i} = y_{i}$, for all 
 $i > l+1$. This implies the fourth part.
 
 The fifth part follows from Lemma \ref{groupoids:40}.
 The last part is immediate from the definitions.
\end{proof}

The last two parts of this result are rather curious: if 
the $R_{n}$ actually formed an increasing sequence of compact, 
open subequivalence relations, then $C^{*}(R)$ would be an 
AF-algebra. The sets $R_{n} \cup R_{n+1}$ \emph{do} form an 
increasing sequence of compact open subsets, but they are 
not subequivalence relations.

\begin{thm}
\label{groupoids:80}
The \'{e}tale equivalence relation $R$ is amenable.
\end{thm}

\begin{proof}
For $ l \geq 1$, we define $h_{l}$ as follows:
\[
h_{l}(x,y) = \left\{ \begin{array}{cc} 
( \# [x]_{R_{l} })^{-1}, & (x,y) \in R_{l}  \\
 0, & (x,y) \notin R_{l} \end{array} \right.
 \]
 This is a continuous function on the space $R$ since $R_{l}$ is 
  \'{e}tale. It is compactly supported as $R_{l}$ is compact 
and it is clearly non-negative. We consider the two conditions 
given in the first section for amenability. The first of these 
 conditions is trivially satisfied.

 We consider $h_{l}$ on $R_{n}$, for $l > n$.
 Let $(x,y)$ be in $R_{n}$.
  If $(x,y)$ actually lies in $R_{l}$ as well, then we have 
  \[
  h_{l}(x,z) = h_{l}(y,z) = ( \# [x]_{R_{l} })^{-1}
  \]
  for every $(x,z)$ in $R_{l}$. 
  Now suppose that 
   $(x,y)$ is not in $R_{l}$. Let $z$ be any point with 
   $(x,z)$ in $R$. If neither $(x,z)$ nor $(y,z)$ are in $R_{l}$, 
   then 
   \[
   h_{l}(x,z) - h_{l}(y,z) = 0 - 0 =0.
   \]
    If $z$ is such that one of 
   $(x,z)$ and $(y,z)$ is in $R_{l}$, then as $(x,y)$ is not in $R_{l}$,
    $(x,z)$ and $(y,z)$ cannot both be 
  so exactly one of them is and we have 
  \begin{eqnarray*}
  \sum_{(x,z) \in R} \vert h_{l}(x,z) - h_{l}(y,z) \vert 
   & = &  \sum_{(x,z) \in R_{l}} \vert h_{l}(x,z) - h_{l}(y,z) \vert \\
      &  &  +  \sum_{(y,z) \in R_{l} } \vert h_{l}(x,z) - h_{l}(y,z) \vert \\
      &  = &   \sum_{(x,z) \in R_{l}} \vert (\# [x]_{R_{l}})^{-1} -0  \vert \\
    &  &  +  \sum_{(y,z) \in R_{l} } \vert 0 -  (\# [y]_{R_{l}})^{-1} \vert \\
    & = & 2.
    \end{eqnarray*}
 
 Now we define $g_{l} = l^{-1} \sum_{i=1}^{l} h_{l}$. This is clearly
 non-negative, continuous, compactly supported and satisfies the first of the 
 two conditions. As for the second, we consider the value 
 of $g_{l}$ on $R_{n}$:
 \begin{eqnarray*}
  \sum_{(x,z) \in R} \vert g_{l}(x,z) - g_{l}(y,z) \vert 
   & \leq & l^{-1} \sum_{i=1}^{n} \sum_{(x,z) \in R} 
   \vert h_{l}(x,z) - h_{l}(y,z) \vert \\ 
   &  &  + l^{-1} \sum_{i=n+1}^{l} \sum_{(x,z) \in R} 
   \vert h_{l}(x,z) - h_{l}(y,z) \vert  \\
   & \leq &  l^{-1} 2 n  + l^{-1}  2 \sum_{i=n+1}^{l} \chi_{R_{n} - R_{i}} \\
    & \leq & l^{-1} 2 (n+1)
\end{eqnarray*}
where we have used the fact that the sets $R_{n} - R_{i}, n < i \leq l,$ are pairwise disjoint from part 4 of Lemma 
 \ref{groupoids:70} in the last step. This clearly tends to 
zero uniformly on $R_{n}$ as $l$ tends to infinity. This completes the proof, using the characterization of amenability given 
in the introduction.
\end{proof}

\begin{rem}
\label{groupoids:90}
We finish this section with an example and a question. Let $\alpha, \beta$ and $\beta'$ be 
irrational numbers so that 
$\Z + \Z \alpha, \Z + \Z \beta, \Z + \Z \beta'$ (with the relative order from the real numbers)
are 
all simple dimension groups.  Let $(V, E), (W, F)$ and $(W', F')$ be Bratteli diagrams for
each with $V_{n}, W_{n}, W'_{n}$ each having two vertices (see \cite{Eff}). After
 replacing $(V, E)$ with some telescope, the diagrams $(W,F)$ and $(W',F')$ can be embedded into $(V, E)$ as above
 and let $R_{\alpha, \beta}$ and $R_{\alpha, \beta'}$ be the 
 \'{e}tale 
 equivalence relations constructed above.
 If we accept Theorem \ref{intro:10} for the moment, in view of the classification theorem and the fact that the 
 irrational rotation $C^{*}$-algebra, $A_{\alpha}$, is also classifiable \cite{ElEv}, 
 we have 
 \[
 C^{*}(R_{\alpha, \beta}) \cong C^{*}(R_{\alpha, \beta'}) \cong A_{\alpha}.
 \]
 I do not know of the groupoids   $R_{\alpha, \beta}$ and $R_{\alpha, \beta'}$
 are isomorphic (as in Definition 2.4 of \cite{GPS1}).
\end{rem}

\section{$K$-theory}
\label{K}

In this section, our single goal is to compute the $K$-theory
of the   $C^{*}$-algebra $C^{*}(R)$. The main tool will be Theorem 2.4 of
\cite{Put3}. Translating the notation from \cite{Put3}, the groupoid $G$ there will 
be our $R$. It will be useful to have some notation 
for the elements of $C^{*}(R)$. As in the last section, For each pair 
$(p,q)$ in
 $I_{n}$, we let $b_{p,q}$ be the 
 charactersitic function of $\beta_{p,q}$. If 
 $(p,q)$ is in
 $I_{n}^{W}$ and $i = 0,1$, we let $d_{p,q}^{i,1-i}$ be the charactersitic 
 function of $\delta_{p,q}^{i,1-i}$.
 
 The $C^{*}$-subalgebra of $C^{*}(R)$ generated by the $b_{p,q}$ 
 is  the AF-algebra $C^{*}(R_{E})$.

Let discuss  relative $K$-groups for $C^{*}$-algebras. If $A$ is a 
$C^{*}$-algebra and $A'$ is a $C^{*}$-subalgebra, the mapping cone 
of the inclusion is 
\[
C(A';A) = \{ f: [0,1] \rightarrow A \mid f \text{ continuous}, 
f(0) = 0, f(1) \in A' \}
\]
which is a $C^{*}$-algebra with point-wise operations.
The relative $K$-theory $K_{i}(A';A)$ is usually defined
as $K_{i+1}(C(A';A)) $, for $i = 0,1$. As there is an exact sequence
\[
0 \rightarrow C_{0}(0,1) \otimes A \rightarrow C(A';A) \rightarrow A' \rightarrow 0
\]
and using the canoncial isomorphism 
$K_{i}( C_{0}(0,1) \otimes A) \cong K_{i+1}(A)$, 
there is a six-term exact sequence
\vspace{.3cm}

\xymatrix{ K_{0}(A') \ar[r] & K_{0}(A) \ar[r] & K_{0}(A';A) \ar[d]  \\
  K_{1}(A';A)\ar[u] & K_{1}(A) \ar[l] &  K_{1}(A')\ar[l] }
\vspace{.3cm}

In \cite{Put2}, a different description of the relative 
group $K_{1}(A';A)$ 
is given. We let $\tilde{A}$ denote the $C^{*}$-algebra 
obtained by adding a unit to 
$A$ and $\tilde{A'} \subseteq \tilde{A}$ is the unital inclusion.
We also let $M_{n}(A)$ denote the 
$n \times n$-matrices 
with entries from $A$. One considers partial 
isometries $v$ in $M_{n}(\tilde{A})$, 
such that $v^{*}v$ is in $M_{n}(\C)$ (the subalgebra generated by 
the unit)
and $v v^{*}$ is in  $M_{n}(\tilde{A'})$. Following \cite{Put2}, 
we let $V_{n}(A';A)$ denote the set of all such elements $v$.
 The relative 
group $K_{1}(A';A)$ is described in terms of equivalence classes of such 
partial isometries, which we denote by $[v]_{r}$ ($r$ for 'relative'). 

One obvious advantage of this description is that two of the maps in the 
six-term exact sequence shown above
 become rather obvious: the one from 
$K_{1}(A)$ into $K_{1}(A';A)$ simply takes $[u]_{1}$, where
$u$ is a 
 unitary in 
$M_{n}(\tilde{A})$, to  $[u]_{r}$ in $K_{1}(A';A)$ 
and the one 
from $K_{1}(A';A)$ to $K_{0}(A')$ takes  $[v]_{r}$ to 
$[vv^{*}]_{0} - [v^{*}v]_{0}$ in $K_{0}(A')$.

We would like to apply the results of \cite{Put3} to our $C^{*}$-algebras
\newline
$C^{*}(R_{E}) \subseteq C^{*}(R)$.
To do so, we must  define the set 
$L \subseteq R$ from page 1492 of \cite{Put3}. 
We let $L$ be the union of all sets $\lambda^{1,0}_{p,q}$, with $(p,q)$  in 
$I_{n}^{W}$  and $n \geq 1$.

Observe that every 
point in $s(\lambda^{1,0}_{p,q})$ has all but finitely many edges 
in $\xi^{1}(F)$, while   every 
point in $r(\lambda^{1,0}_{p,q})$ has all but finitely many edges 
in $\xi^{0}(F)$. It follows that $r(L) \cap s(L)$ is empty.
 It is clear then that 
$L^{-1}$ is simply the union of all $\lambda^{0,1}_{p,q}$.

We know from 
Lemma \ref{groupoids:30} that, for any such $p,q$, we have
 $\delta^{1,0}_{p,q} - R_{E} = \lambda^{1,0}_{p,q}$. 
 Also, as $R_{E}$ is the union of all
 $\beta_{p,q}$, we know that $\beta_{p,q} - R_{E}$ is empty. 
 It follows then that the 
 complement of $R_{E}$ in $R$ is exactly $L \cup L^{-1}$.
 As the equivalence relation $R_{E}$ is  open in $R$, it
  follows that $L$ is closed.

 Following \cite{Put3}, we let 
 \[
 H_{0} = L^{-1}L , H_{1} =  LL^{-1}, H' = H_{0} \cup H_{1}, 
 H = H_{0} \cup H_{1} \cup L \cup L^{-1}.
 \]
 This defines these objects simply as subgroupoids. 
 They are not locally compact in general. The main 
 point of \cite{Put3} is that each can be endowed 
 with natural new topology which is finer 
 than the relative topology from $R$ and is \'{e}tale. 
  
 In view of Lemma \ref{groupoids:30},
  we see that $\lambda^{1,0}_{p,q} = \delta^{1,0}_{p,q} \cap L$
 and this means that the sets $\lambda^{1,0}_{p,q}$ are all clopen in the 
 relative topology of $L$. It follows that, for $(p,q)$ in $I_{n}^{W}$, $n \geq 1$,
  the sets
 \[
 \alpha^{0,0}_{p,q} = \lambda_{p,p}^{0,1} \circ \lambda_{p,q}^{1,0} = \lambda_{p,q}^{0,1} \circ \lambda^{1,0}_{q,q}
 \]
  are compact and 
 open and a base for the topology of  $H_{0}$.
 We also define 
  \[
 \alpha^{1,1}_{p,q} = \lambda_{p,p}^{1,0} \circ \lambda_{p,q}^{0,1} = \lambda_{p,q}^{1,0} \circ \lambda^{0,1}_{q,q}
 \]
 which  are compact and 
 open and a base for the topology of  $H_{1}$. The set $L \cup L^{-1}$ keeps its 
 topology from $R$ and $H$ is given the disjoint union topology.

  For each $(p,q)$ in $I_{n}^{W}$,
  we define $a^{i,i}_{p,q}$ to be the characteristic function
 of $\alpha^{i,i}_{p,q}$ which is in $C^{*}(H_{i}) \subseteq C^{*}(H)$.
 The span of these elements 
 is dense in $C^{*}(H_{i})$.
 We also let $a^{1,0}_{p,q}$ be the characteristic function of 
 $\lambda^{1,0}_{p,q}$ and $a^{0,1}_{p,q}$ be the characteristic function of 
 $\lambda_{p,q}^{0,1}$, which are both in $C^{*}(H)$.

 The following is essentially the same as Lemma 3.8 of 
 \cite{Put3} and so we do not provide a proof. On the 
 other hand, it is a simple 
 matter to see that the $a^{i,j}_{p,q}$ form a systems
  of matrix units 
 for the various finite-dimensional algebras and this 
 can be proved directly using
 arguments  following Definition \ref{groupoids:20}.
 It does seem worthwhile to point out that for 
 any $n \geq 1$, $(p,q)$ in $I_{n}^{W}$, $i,j=0,1$
 and $m > n$, we have 
 \[
 a_{p,q}^{i,j} = \sum_{p' \in F_{n,m}, \xi(i(p')) = t(p)}
  a_{p\xi^{i}(p'),q\xi^{j}(p')}^{i,j}, 
  \] 
  since this is a slightly different relation than 
  satisfied by the $b_{p,q}$.

 \begin{thm}
 \label{K:10}
 \begin{enumerate}
 \item The $C^{*}$-algebras $C^{*}(H_{0}), C^{*}(H_{1})$ and 
 $C^{*}(H)$ are all AF-algebras.
 \item We have 
 \[
 K_{0}(C^{*}(H_{0})) \cong  K_{0}(C^{*}(H)) \cong K_{0}(C^{*}(H_{1}))
  \cong G_{1}
  \]
  and the first two isomorphism are induced by the inclusion
  maps.
  \item The map from $ K_{0}(C^{*}(H_{0}) )$ to 
  $K_{1}(C^{*}(H'); C^{*}(H)) $ which sends $[ a^{0,0}_{p,p}]_{0}$ to 
  $\left[ \begin{array}{cc} a^{1,0}_{p,p} & 0 \\ 1 - a^{0,0}_{p,p} & 0 
  \end{array}
 \right]_{r}$
   is an isomorphism.
  \end{enumerate}
  \end{thm}
  
    The important relation between the $C^{*}(H)$ and 
    our earlier $C^{*}(R)$  is the fact that 
  elements of $C^{*}(R)$ act as multipliers of $C^{*}(H)$. 
  When considering  
  those elements which are continuous functions of compact support 
  on $R$, this is  a matter of realizing that
   $H \subseteq R$ and one simply restricts the functions. The
    topology which $H$ has been given is not the  relative topology
    from $R$, but it is finer  so that a continuous function
    on $R$ restricts to a continuous function on $H$. If it
    has compact support on $R$, its support will no longer 
    be compact on $H$, but the function will be bounded.
    
We need to perform this product in  a few simple cases.   
    
\begin{lemma}
\label{K:15}
Let $ m \geq 1$, $(p,p')$ be in $I_{m}$, $q$ in $E_{0,m}$ and $i,j = 0,1$.
We have 
\begin{enumerate}
\item 
\[
d_{p,p'}^{i,1-i}  a_{q,q}^{j,j} = \left\{ \begin{array}{cl}
 a^{i,1-i}_{p,p'} & \text{ if } 1-i=j, p'=q \\
  0 & \text{ otherwise } \end{array} \right.
\]
\item 
\[
a_{q,q}^{j,j} d_{p,p'}^{i,1-i}  = \left\{ \begin{array}{cl}
 a^{i,1-i}_{p,p'} & \text{ if } i=j, p=q \\
  0 & \text{ otherwise } \end{array} \right.
\]
\item 
\[
b_{p,p'}  a_{q,q}^{j,j} = \left\{ \begin{array}{cl}
 a^{j,j}_{p,p'} & \text{ if }  p'=q \\
  0 & \text{ otherwise } \end{array} \right.
\]
\item
\[
 a_{q,q}^{j,j} b_{p,p'}  = \left\{ \begin{array}{cl}
 a^{j,j}_{p,p'} & \text{ if }  p=q \\
  0 & \text{ otherwise } \end{array} \right.
\]
\end{enumerate}
\end{lemma}

\begin{proof}
The proof of the first part   involves computing the 
composition
\[
\delta_{p,p'}^{i,1-i} \circ \alpha^{j,j}_{q,q} 
= \delta_{p,p'}^{i, 1-i}  \circ \lambda_{q,q}^{j,1-j} 
\circ \lambda_{q,q}^{1-j,j}.
\]
It follows from the defintions that this is empty unless
$j=1-i$ and $p'=q$. In that case, it is simply
$\lambda_{p,p'}^{i,1-i}$. This completes the proof. 
The second part is done in a similar way and we omit the details.
The  others parts can be obtained from these
by taking adjoints.
\end{proof}

\begin{lemma}
\label{K:18}
For each $m \geq 1$, define
 \[
  e_{m} = \sum_{q \in E_{0,m}, t(q) \in \xi(W) } 
  a^{0,0}_{q,q} + a^{1,1}_{q,q},
  \]
  in $C^{*}(H_{0}) \oplus C^{*}(H_{1})$. 
  \begin{enumerate}
  \item 
  The sequence $e_{m}, m \geq 1$ is an approximate identity for
  $C^{*}(H)$.
  \item 
  If $n < m$ and $(p,q)$ is in $I_{n}^{W}$, then 
  \[
  b_{p,q} e_{m} = e_{m} b_{p,q},
  \]
  is in $C^{*}(H_{0}) \oplus C^{*}(H_{1})$.
  \item 
  If $n < m$ and $(p,p)$ is in $I_{n}^{W}$, then 
  \[
  d_{p,p}^{1,0} e_{m} = e_{m} d_{p,p}^{1,0} = a_{p,p}^{1,0} + c,
  \]
  where $c$ is a partial isometry in $C^{*}(H_{0}) \oplus C^{*}(H_{1})$. 
  Moreover, $a_{p,p}^{1,0}$ and $c$ have orthogonal ranges and orthogonal 
  sources; that is, $c^{*}a_{p,p}^{1,0} = 0 = a_{p,p}^{1,0} c^{*}$.
  \end{enumerate}
\end{lemma}
 
\begin{proof}
The proof of the first part is straightforward and we omit it.
For the second part, we use part 5 of Lemma \ref{groupoids:10} 
to write  $b_{p,p'}$ as a sum of $b_{q,q'}$ over pairs $(q,q')$ 
in $I_{m}$. We also write $e_{m}$ as  a sum of $a_{q,q}^{i,i}$ as
 in its definition. We then expand out the product $b_{p,q} e_{m}$, applying 
part 2 of Lemma \ref{K:15} to each term. A similar approach to 
$e_{m} b_{p,q}$ yields the result.

  For the third part, 
   repeated application of Lemma \ref{groupoids:40} means that 
    we may write 
   \[
   d_{p,p}^{1,0} = \sum_{p' \in F_{n,m}, \xi(s(p')) = t(p)}
    d_{p\xi^{1}(p'),p\xi^{0}(p')}^{1,0} + b,
   \]
   where $b$ is a sum of terms of the form $\beta_{q',q''}$ with $(q',q'')$
   in $I_{m+1}$.  Moreover, these are two partial isometries with orthogonal
   initial projections and final projections.
   We now take the product with $e_{m}$, again writing it as 
   a sum of $a_{q,q}^{i,i}$ and expanding. By using Lemma \ref{K:15}, we
   we first see that $d_{p,p}^{1,0} e_{m} = e_{m} d_{p,p}^{1,0}$.
   For a fixed $p' \in F_{n,m}, \xi(s(p')) = t(p)$, we also have 
   \begin{eqnarray*}
   d_{p\xi^{1}(p'),p\xi^{0}(p')}^{1,0} e_{m} 
  & = & \sum_{q} d_{p\xi^{1}(p'),p\xi^{0}(p')}^{1,0} 
  (a^{0,0}_{q,q} +  a_{q,q}^{1,1}) \\
     & = & a_{p\xi^{1}(p'),p\xi^{0}(p')}^{1,0}, 
   \end{eqnarray*} 
   using Lemma \ref{K:15}. Taking the sum over all $p'$ now yields
   \[
  \left( \sum_{p' \in F_{n,m}, \xi(s(p')) = t(p)}
    d_{p\xi^{1}(p'),p\xi^{0}(p')}^{1,0} \right) e_{m} = a_{p,p}^{1,0}.
   \] 
   The term $b e_{m} = c$ lies in $C^{*}(H_{0}) \oplus C^{*}(H_{1})$ 
   in consequence of Lemma \ref{K:15}.
This completes the proof.   
\end{proof}

 \begin{prop}
 \label{K:20}
 With the isomorphism 
 \[
 \alpha: K_{1}( C^{*}(R_{E}); C^{*}(R)) \rightarrow K_{1}(C^{*}(H'); C^{*}(H))
 \]
 given in Theorem 3.1 of 
 \cite{Put2}, for any $n \geq 1$ and 
 $p$ in $E_{0,n}$ with $t(p)$ in $\xi(W)$, we have 
 \[
 \alpha\left[ \begin{array}{cc} d_{p,p}^{1,0} & 0 \\ 1 - (d_{p,p}^{1,0})^{*}d_{p,p}^{1,0} & 0 \end{array} \right]_{r}
  = \left[\begin{array}{cc} a_{p,p}^{1,0} & 0 \\ 1 - a_{p,p}^{0,0} & 0 \end{array} \right]_{r}.
 \]
 \end{prop}
 
 \begin{proof}
 The map $\alpha$ is constructed as follows. One finds 
 an approximate unit, $e_{t}, t \in [0, \infty)$,  for the 
 algebra $C^{*}(H)$. Then for a partial isometry $v$ in $M_{n}(C^{*}(R)^{\sim})$, 
 such that $v^{*}v$ and $vv^{*}$ are in 
  $M_{n}(C^{*}(R_{E})^{\sim})$, we form 
  \[
  \alpha(v)_{t} = \left[ \begin{array}{cc} ve_{t} & 0 \\
  \left( v^{*}v - e_{t}v^{*}ve_{t} \right)^{1/2} & 0 \end{array} \right].
  \]
  Notice here that we write $e_{t}$ rather than 
  $1_{n} \otimes e_{t} \in M_{n}(C^{*}(H))$, 
  for simplicity.
  For sufficently large values of $t$, this defines a class 
  in  \newline
  $K_{1}(C^{*}(H_{0}) \oplus C^{*}(H_{1}); C^{*}(H))$ which is the image of
  $[v]_{r}$ under $\alpha$.

  We extend our approximate unit $e_{m}$ extend to real values 
  by setting $e_{t} = (m+1-t) e_{m} + (t-m) e_{m+1}$, for 
  $ m \leq t \leq m+1$.
  
It will be convenient to denote
\[
b = (d_{p,p}^{1,0})^{*}d_{p,p}^{1,0} = 
\sum_{f \in F_{n+1}, \xi(i(f)) = t(p)} b_{p\xi^{0}(f), p \xi^{0}(f)}.
\]

Let $v = \left( \begin{array}{cc} d_{p,p}^{1,0} & 0 \\
 1 - b & 0 \end{array} \right)$.
Then, for $m > n$,  using the facts that $e_{m}$ 
and $b$   are projections and commute, we have 
  \begin{eqnarray*}
  \alpha(v)_{m} & = &  \left[ \begin{array}{cc} ve_{m} & 0 \\
  \left( v^{*}v - e_{m}v^{*}ve_{m} \right)^{1/2} & 0 \end{array} \right]  \\
     &  =  &
   \left[ \begin{array}{cccc} d_{p,p}^{1,0} e_{m} & 0 & 0 & 0 \\
     ( 1 - b) e_{m} & 0 & 0 & 0 \\
  ( 1- e_{m} ) & 0 & 0 & 0  \\
  0 & 0 & 0 & 0 
   \end{array} \right] \\
    &  =  &
   \left[ \begin{array}{cccc} a_{p,p}^{1,0} + a' & 0 & 0 & 0 \\
     ( 1 - b) e_{m} & 0 & 0 & 0 \\
  ( 1- e_{m} ) & 0 & 0 & 0  \\
  0 & 0 & 0 & 0 
   \end{array} \right] 
  \end{eqnarray*}
  
  We will use part (iii) of Lemma 2.2 of \cite{Put2}, which states that 
  left multiplication by a unitary in
   $M_{4}( \left( C^{*}(H_{0}) \oplus C^{*}(H_{1})\right)^{\sim})$ 
  does not change the class of this element in the relative $K$-group.
  We then note that 
  \[
   \left[ \begin{array}{cccc} 1 & 0 & 0 & 0 \\
     0 & e_{m} & 1-e_{m} & 0 \\
  0 & 1-e_{m} & e_{m} & 0  \\
  0 & 0 & 0 & 1 
   \end{array} \right] 
    \left[ \begin{array}{cccc} a_{p,p}^{1,0} + a' & 0 & 0 & 0 \\
     ( 1 - b) e_{m} & 0 & 0 & 0 \\
  ( 1- e_{m} ) & 0 & 0 & 0  \\
  0 & 0 & 0 & 0 
   \end{array} \right] =
    \left[ \begin{array}{cccc} a_{p,p}^{1,0} + a' & 0 & 0 & 0 \\
      1 - b e_{m} & 0 & 0 & 0 \\
  0 & 0 & 0 & 0  \\
  0 & 0 & 0 & 0 
   \end{array} \right] 
  \]
  Then, we compute, noting that $c$ is in $  C^{*}(H_{0}) \oplus C^{*}(H_{1})$,
 \[
   \left[ \begin{array}{cccc} 1-cc^{*} & c & 0 & 0 \\
      c^{*} & 1 - c^{*}c & 0 & 0 \\
  0 & 0 & 1 & 0  \\
  0 & 0 & 0 & 1 
   \end{array} \right]
   \left[ \begin{array}{cccc} a_{p,p}^{1,0} + c & 0 & 0 & 0 \\
      1 - b e_{m} & 0 & 0 & 0 \\
  0 & 0 & 0 & 0  \\
  0 & 0 & 0 & 0 
   \end{array} \right] =
     \left[ \begin{array}{cccc} a_{p,p}^{1,0}  & 0 & 0 & 0 \\
      1 - a_{p,p}^{0,0} & 0 & 0 & 0 \\
  0 & 0 & 0 & 0  \\
  0 & 0 & 0 & 0 
   \end{array} \right]
  \]   
 \end{proof}
 
 \begin{prop}
 \label{K:30}
 In the six-term exact sequence of Theorem 2.4 of \cite{Put3}, the map 
 from $K_{0}(C^{*}(H))$ to $K_{0}(C^{*}(R_{E}))$ is zero.
 \end{prop}
 
 \begin{proof}
 The map is question is the composition of the isomorphism \newline
 $K_{0}(C^{*}(H)) \cong K_{0}(C^{*}(H_{0})) $, the 
 isomorphism \newline
  $ K_{0}(C^{*}(H_{0}) ) \cong K_{1}(C^{*}(H'); C^{*}(H)) $
 of Theorem \ref{K:10} and inverse of  the map $\alpha$  of 
 Proposition \ref{K:20}.
 We begin $[a^{0,0}_{p,p}]_{0}$, with  $p$ in $E_{0,n}$ with $t(p) = \xi(w)$, 
 $w\in W$.
By Theorem \ref{K:10}, the second isomorphism sends this to 
 $\left[ \begin{array}{cc} a^{1,0}_{p,p} & 0 \\
   1 - a^{0,0}_{p,p} & 0 \end{array} \right]_{r}$ 
   in $K_{1}(C^{*}(H'); C^{*}(H)) $. 
 By Proposition \ref{K:20}, the inverse of $\alpha$ carries this to
  $\left[ \begin{array}{cc} d^{1,0}_{p,p} & 0 \\
   1 - (d^{1,0}_{p,p})^{*}d^{1,0}_{p,p}  & 0 \end{array} \right]_{r}$ .
   Finally, the map from $K_{1}( C^{*}(R_{E}); C^{*}(R))$ 
 to $K_{0}(C^{*}(R_{E}))$ sends this to 
 \[
 \left[ \begin{array}{cc} d^{1,0}_{p,p}(d^{1,0}_{p,p})^{*} & 0 \\
  0 &  1 - (d^{1,0}_{p,p})^{*}d^{1,0}_{p,p}  \end{array} \right]_{0}
  - \left[ \begin{array}{cc} 1 & 0 \\
   0 & 0 \end{array} \right]_{0}
    = [ d^{1,0}_{p,p}(d^{1,0}_{p,p})^{*}]_{0} - 
    [ (d^{1,0}_{p,p})^{*}d^{1,0}_{p,p}]_{0}.
 \]
 As $d^{1,0}_{p,p}$ is the characteristic function of $\delta^{1,0}_{p,p}$, 
 $d_{p,p}^{1,0}(d_{p,p}^{1,0})^{*} $ is the characteristic function of the 
 source of $\delta^{1,0}_{p,p}$ while $(d_{p,p}^{1,0})^{*}d^{1,0}_{p,p}$ 
 is the characteristic function of the 
 range of $\delta^{1,0}_{p,p}$. 
 The former is the set $U^{1}(p)$, while the latter is $U^{0}(p)$. 
 We let   
 \[
 w = \cup_{i(f) = w} b_{p\xi^{1}(f), p\xi^{0}(f)}.
 \]
 This is clearly in $R_{E}$ and it is a simple matter to see that 
 \[
 ww^{*} = d_{p,p}^{1,0}(d_{p,p}^{1,0})^{*}, 
 w^{*}w = (d_{p,p}^{1,0})^{*}d_{p,p}^{1,0}.
 \]
 We conclude that the two projections 
 in $C^{*}(R_{E})$ are equivalent in $C^{*}(R_{E})$.
  This completes the proof.
 \end{proof}

We are now ready to prove the third and fourth
 parts of the main theorem,
\ref{intro:10} .

We consider the six-term exact sequence given in Theorem 2.4 of \cite{Put3}.

   \vspace{.3cm}

\xymatrix{ K_{0}(C^{*}(H))  \ar[r] & K_{0}(C^{*}(R_{E})) \ar[r] &
 K_{0}(C^{*}(R)) \ar[d] \\
 K_{1}(C^{*}(R)) \ar[u] & K_{1}(C^{*}(R_{E}))  \ar[l] & 
  K_{1}(C^{*}(H))   \ar[l] }
     \vspace{.3cm}

  We know that $C^{*}(H)$ and $C^{*}(R_{E})$ are both AF-algebras and their 
  $K$-one groups are trivial. In addition, we know that 
  $K_{0}(C^{*}(R_{E})) \cong G_{0}$ and $K_{0}(C^{*}(H)) \cong G_{1}$. 
  Putting in all of this information, we have 
     \vspace{.3cm}

  \xymatrix{ G_{1}  \ar[r] & G_{0} \ar[r] &
 K_{0}(C^{*}(R)) \ar[d] \\
 K_{1}(C^{*}(R)) \ar[u] & 0 \ar[l] & 
  0  \ar[l] }
     \vspace{.3cm}

  Finally, the last Proposition tells us that the horizontal map 
  from \newline 
  $K_{0}(C^{*}(H)) \cong G_{1}$ is zero. It follows that the two remaining 
  maps are isomorphisms.
  
  We must still show that the map on $K_{0}$ is an order isomorphism.
  As it arises from the inclusion of $C^{*}(R_{E}) $ in $C^{*}(R)$,
  it sends positive elements to positive elements. Conversely, suppose that
  an element in $K_{0}(C^{*}(R_{E}))$ has a positive image in 
  $K_{0}(C^{*}(R))$. Then the image under that element under every trace
  on $C^{*}(R)$ is strictly positive. On the other hand, these 
  traces all arise from $R$-invariant measures on $X_{E}$, which 
  are precisely the $R_{E}$-invariant measures by \ref{groupoids:60}. 
  As $C^{*}(R_{E})$ is a simple AF-algebra, it follows that the 
  element is positive in  $K_{0}(C^{*}(R_{E}))$.

  \begin{rem}
 \label{K:40}
 For readers who might be disappointed by  the fact that 
  our computation of  $K_{1}(C^{*}(R))$ did not 
   explicitly produce a single unitary in $C^{*}(R)$, 
   let us note the following. If we fix a path $p$ 
   with $t(p) $ in $\xi(W)$, and define
   \[
   v_{p} =  \left( \sum_{f \in F_{n+1}, \xi(i(f)) = t(p)}
   b_{p\xi^{0}(f), p \xi^{1}(f)} \right) d^{1,0}_{p,p},
   \]
   then $v_{p}$ is a partial isometry in $C^{*}(R)$
   and 
   \[
   v_{p}^{*}v_{p} = v_{p} v_{p}^{*} = \sum_{f \in F_{n+1}, \xi(i(f)) = t(p)}
   b_{p\xi^{0}(f), p \xi^{0}(f)} \leq b_{p,p},
   \]
    and $v_{p} + (1 - v_{p}^{*}v_{p})$ is a unitary.
    Under the natural map from $K_{1}(C^{*}(R))$ to 
    $K_{1}(C^{*}(R_{E}),C^{*}(R))$, it is sent to 
    $\left[ \begin{array}{cc} d_{p,p}^{1,0} & 0 \\ 1 - (d_{p,p}^{1,0})^{*}d_{p,p}^{1,0} & 0 \end{array} \right]_{r}$.
    
    Pursuing this slightly further,
     it is not difficult to show that
     \[
  C^{*}(B_{t(p)}, v_{p} + (b_{p,p} - v_{p}^{*}v_{p})) \cong C(\T) \otimes B_{t(p)},
  \]
  where $B_{t(p)}$ is as described earlier and $\T$ denotes the circle.
  Repeating this construction at every vertex of $\xi(W_{n})$, we 
  can construct a $C^{*}$-subalgebra of $C^{*}(R)$ which 
  we denote by $D_{n}$, containing $B_{n}$, and with 
  \[
  D_{n} \cong 
  \left( \oplus_{w \in W_{n}} C(\T) \otimes B_{\xi(w)} \right)
   \oplus  \left( \oplus_{v \in V_{n} - \xi(W_{n})}  B_{v} \right).
   \]
   This looks remarkably like an approximating subalgebra
   with the correct $K$-theory. 
   
   There is some difficulty however. Letting $f_{n+1}, f_{n+2}, \ldots$
   be any path in $(W,F)$ with $\xi(i(f_{n+1}))=v$, the function
   $v_{p}$ takes value one on the point
   \[
   ((p, \xi^{0}(f_{n+1}), \xi^{1}(f_{n+2}), 
   \xi^{1}(f_{n+3}), \ldots ), (p, 
   \xi^{0}(f_{n+1}), \xi^{0}(f_{n+2}), 
   \xi^{0}(f_{n+3}), \ldots )).
   \]
    In fact, all 
   the generators we have listed for all the subalgebras $D_{l}$ with $l > n$
   are zero on this point, so no linear
    combination of them can approximate $v_{p}$.
    (This is related to part 4 of Lemma \ref{groupoids:70}.)
     However, the supports of the generators for 
   $D_{l}$ are contained in a compact set, but not in a
    compact  subequivalence relation, so it is possible for some other 
   element of 
   $D_{l}$ to approximate $v_{p}$. It seems problematic to write such an 
   element
   explicitly without reference to some deep techniques, such as Berg's 
   technique.
   
   This means that it is not clear (at least to the author)
   that some $D_{l}, l > p$ will contain or 
   approximately contain $v_{p}$, nor is it clear that the union
   of the $D_{l}$ will be dense in $C^{*}(R)$.
  \end{rem}

  Let us finish this section with a pair of relatively 
  simple applications 
   of Theorem 2.4 from 
  \cite{Put3}. 
  
  \begin{rem}
  \label{K:50} We describe a minimal equivalence relation, $\tilde{R}$,
  on a Cantor set, $X$, with two distinct \'{e}tale topologies.
  It will contain a subequivalence relation $R$ and subset $L$ satisfying 
  the appropriate conditions from \cite{Put3}.  The two different topologies
  agree on both $R$ and $L$. In fact, $R$ is an AF-equivalence relation.
  This means that $H$ and $H'$ are exactly the same for the 
  two situations. The  place where the difference in the topologies 
  appears is in the map from $K_{0}(C^{*}(H))$ to $K_{0}(C^{*}(R))$. 
  (This is the same map we took some care to compute 
  in Proposition \ref{K:30}.) In fact, the two different topologies 
  produce different $C^{*}$-algebras.
  
  We consider a Cantor set $X$ with a minimal homeomorphism 
  $\varphi$ of $X$. Let $R_{\varphi}$ denote the orbit relation of $\varphi$
  with the natural topology described in section \ref{groupoids}.
  That is, each equivalence class is $\{ \varphi^{n}(x) \mid n \in \Z \}$, 
  for some $x$ in $X$.
   We select two points $y, z$ having distinct orbits and 
   let $R$ be the subequivalence relation with exactly the same 
   equivalence classes except $\{ \varphi^{n}(y) \mid n \leq 0 \},
   \{ \varphi^{n}(y) \mid n \geq 1 \}, \{ \varphi^{n}(z) \mid n \leq 0 \}$
  and  $\{ \varphi^{n}(z) \mid n \geq 1 \}$ are all 
  distinct equivalence classes \cite{Put1}.
  
  Also, let $R_{1}$ be the subequivalence relation with exactly the same 
   equivalence classes as $R_{\varphi}$ 
   except $ \{ \varphi^{n}(z) \mid n \leq 0 \}$
  and  $\{ \varphi^{n}(z) \mid n \geq 1 \}$ are  distinct equivalence classes.
  It follows from the results of \cite{Put1} that both $R$ and $R_{1}$
   are  AF-equivalence relations.
  
   Next, draw a Bratteli diagram for $R$, so we may identify $X$ with the path 
   space of this diagram. Put
    an order on the edges (as in \cite{HPS}) such that $y$ is the unique 
    infinite maximal path, while $\varphi(y)$ is the unique 
    infinite minimal path. Let $\psi$ be the associated Bratteli-Vershik map
    and $R_{\psi}$ be  its orbit relation with the natural topology.
    
    It is an easy matter to see that $R_{\psi}$ and $R_{1}$ are actually the 
    same equivalence relation, which we called $\tilde{R}$ above.
     They carry different topologies, however. 
    The latter is AF, while the former is not. Both contain 
    $R$ as an open subequivalence relation and 
    the two topologies agree on $R$.
     In each case, we may apply the 
    results of \cite{Put3} using 
    \[
    L = \{ (\varphi^{n}(y), \varphi^{m}(y) ) \mid n \leq 0, m \geq 1\}.
    \]
    The set $L$ is discrete in both topologies and in both cases
    the groupoids $H_{0}, H_{1}$ and $H$ are the same. It is a simple 
    matter to check that
    $C^{*}(H) \cong \mathcal{K}$, the $C^{*}$-algebra of compact operators,
    and $K_{0}(C^{*}(H)) \cong \Z$. If we consider the 
    six-term exact sequences in Theorem 2.4 of \cite{Put3}, one involving
    $C^{*}(R), C^{*}(R_{1})$ and $C^{*}(H)$ and the 
    other  involving
    $C^{*}(R), C^{*}(R_{\psi})$ and $C^{*}(H)$, four of the
     terms are identical:
     \vspace{.3cm}

\xymatrix{\Z \ar[r] &  K_{0}(C^{*}(R)) \ar[r] & K_{0}(C^{*}(R_{1})) \ar[d] 
  \\ 
 K_{1}(C^{*}(R_{1})) \ar[u] &  0 \ar[l] &  0 \ar[l] 
 \\
 \Z \ar[r] &   K_{0}(C^{*}(R)) \ar[r] & K_{0}(C^{*}(R_{\psi})) \ar[d]  \\
  K_{1}(C^{*}(R_{\psi})) \ar[u]  & 0 \ar[l] &  0 \ar[l]}
\vspace{.3cm}

    The difference arises in the maps between them. In the first case, 
    \newline
    $K_{1}(C^{*}(R_{1}))=0$ since it is an AF-algebra 
    and the $\Z$ in the upper left \newline
    corner maps injectively under the  
    map to its right. In the second case,
     $K_{1}(C^{*}(R_{\psi})) \cong \Z$ and the vertical map into 
     $\Z$ is an isomorphism.
     \end{rem}

\section{Tracially AF}
\label{TAF}

The aim of this section is to prove that the $C^{*}$-algebra
$C^{*}(R)$ is tracially AF. 

Recall that we begin with  $\mathcal{F}$ in $C^{*}(R)$, $\epsilon > 0$ 
and non-zero projection $p_{0}$ in $C^{*}(R)$. We must show that
there 
is a finite-dimensional $C^{*}$-subalgebra
$F \subseteq  A$ with unit
$p$  satisfying
\begin{enumerate}
\item 
$\Vert pa - ap \Vert < \epsilon$, for all  $a$ in $\mathcal{F}$, 
\item 
$ pap \in_{\epsilon} F$, for all $x$ in $\mathcal{F}$, 
\item 
$1-p$ is unitarily equivalent to a subprojection of 
$p_{0}$.
\end{enumerate}

We know that $C^{*}(R)$ is the closed linear span 
of the elements $b_{p,q}$, with $(p,q)$ in 
$I_{n}$,
 along with $d_{p,q}^{i,1-i}$, over $(p,q)$ in 
$I_{n}^{W}$ and  $i = 0,1$.

 Given the presence of the $\epsilon$ in the definition, it 
 suffices for us to consider $\mathcal{F}$ to be a finite subset 
 of such elements and then we may as well assume that 
 $\mathcal{F}$ consists of all such $b_{p,q}, (p,q) \in I_{n}$ 
 and $d_{p,q}^{i,1-i}, (p,q) \in I_{n}^{W}$, for  
 some finite range $1 \leq n \leq n_{0}$. We may also assume that 
 $n_{0}$ is chosen so that the span of these elements contains 
 a positive contraction  $a$ with $\Vert a - p \Vert < \epsilon$.

 For $l > n_{0}$ and $i=0, 1$, let $Q^{i}_{l}$ denote the set of all 
 paths $q$ in $E_{0,l}$ such that $(q_{n_{0}+1}, \ldots q_{l})$ is 
 in $\xi^{i}(F_{n_{0},l})$. 
Define $Y^{i}_{l}$ to be the set of all $x$ in 
$X_{E}$ such that  $(x_{n_{0}+1}, \ldots, x_{l})$ is in $ Q^{i}_{l}$.

We also define
 \[
 \bar{e}^{i}_{l} = \sum_{q \in Q^{i}_{l}} b_{p,p},  \hspace{1cm}
 \bar{e}_{l} = \bar{e}^{0}_{l} + \bar{e}^{1}_{l},
 \]
 which is a projection in $C^{*}(R_{E})$. We 
 will use $p = 1 - \bar{e}_{l}$, for suitably chosen $l > n_{0}$. 
 (We use $\bar{e}_{l}$ so that it is not confused with 
   the $e_{m}$ of 
 the last section.)
 
 Our first result shows that the projections
  $\bar{e}_{l}$ tend to zero, in two different senses.

\begin{lemma}
\label{TAF:5}
\begin{enumerate}
\item
For any trace $\tau$ on $C^{*}(R_{E})$, we have 
\[
\tau(\bar{e}_{l}) \leq 2^{1+n_{0}-l}.
\]
\item 
There is a faithful, unital representation $\pi$ of $C^{*}(R)$
such that $\pi(\bar{e}_{l}), l =1, 2, \ldots,$ 
converges to zero in the strong operator topology.
\end{enumerate}
\end{lemma}

\begin{proof}
We know that $\tau(\bar{e}_{l}) = \tau(\bar{e}_{l}^{0}) + \tau(\bar{e}_{l}^{1}) = \mu(Y_{l}^{0}) + \mu(Y_{l}^{1})$, 
for some $R_{E}$-invariant measure on $X_{E}$.
 
 Fix a path $q$ in $E_{0,n_{0}}$ and $i=0, 1$.
Let $v$ be a  vertex in $V_{l}$ and let $Q'_{v}$ be the set of all 
paths $q'$ in $E_{n_{0}+1, l}$ with $t(q') = v$. At each individual edge set,
 at most one half of the edges  between two fixed vertices are in 
 $\xi^{0}(F)$ (since there are an equal number in 
  $\xi^{1}(F)$ and they are disjoint). It follows that 
  the number of paths in $Q'_{v}$ which lie in $\xi^{i}(F)$ are at most 
  $2^{n_{0}-l}$ of the total number in $Q'_{v}$. For all of these paths $q'$, $U(qq')$
  all have the same measure. If we apply the measure $\mu$ and sum over the vertices, we get
 \[
\mu(U(q) \cap Y_{l}^{i}) \leq 2^{n_{0}-l} \mu(U(q)).
\] 
Now summing over all paths $q$ and $i=0,1$ yields the first result.

For the second part, let $(f_{1}, f_{2}, \ldots)$ be any path in $X_{F}$ and let 
\[
x = (\xi^{1}(f_{1}), \xi^{0}(f_{2}), \xi^{1}(f_{3}), \xi^{0}(f_{2}), \ldots)
\]
be in $X_{E}$. Consider the left regular representation, $\lambda_{x}$,
 of $C^{*}(R)$
on $\ell^{2}[x]_{R} = \ell^{2}[x]_{R_{E}}$ (\cite{Ren}). 
We know this is faithful because $C^{*}(R)$ is simple.
 Consider the standard basis 
for $\ell^{2}[x]_{R_{E}}$: a typical element is supported on a single
point $y$ with $(x, y)$ in $R_{E}$. We denote this function by $\delta_{y}$.
This means that there is some $i_{0}$ 
such that
$y_{i} = x_{i}$, for all 
$i \geq i_{0}$. If $l > i_{0}+2$, then $\lambda_{x}(\bar{e}_{l}) \delta_{y} =0$.
The desired conclusion follows.
\end{proof}

 Let us make some easy observations. First, it follows easily from
 Lemma \ref{groupoids:10} that we have 
 \[
 \bar{e}_{l} b_{p,p'} = b_{p,p'} \bar{e}_{l} 
 \]
 for all $(p,p')$ in $I_{n}$ with $1 \leq n < l$
  and that the product lies in 
 finite-dimensional $C^{*}$-algebra
$B_{l}$ as described following Definition 
\ref{groupoids:20}.
Next, we consider similar facts for
products with the $d^{i,1-i}_{p,p'}$.

 \begin{lemma}
 \label{TAF:10}
 For all $n < l$, $(p,p')$ in $I_{n}$ and $i = 0,1$, we have 
 \[
 \bar{e}_{l} d^{i, 1-i}_{p,p'} =  d^{i, 1-i}_{p,p'}  \bar{e}_{l}
 \] 
 and 
 \[
 ( 1 - \bar{e}_{l})  d^{i, 1-i}_{p,p'} \in B_{l}.
 \]
 \end{lemma}
 
 \begin{proof}
 The proof will be by induction on $l-n$. In the case, 
 $l-n=1$ or $l = n+1$, we will actually prove that
 $\bar{e}_{l} d^{i,1-i}_{p,p'}  = d^{i,1-i}_{p,p'}$. For 
 simplicity, we will assume that $i=1$.
 
 It is a consequence of  Lemma \ref{groupoids:40} that we may write
\begin{eqnarray*}
d^{1,0}_{p,p'} & =  & \sum_{f \in A} d^{1,0}_{p\xi^{1}(f), p'\xi^{0}(f)} \\
 &  & + \sum_{(f,f') \in B} b_{p\xi^{1}(f)\xi^{0}(f'), q\xi^{0}(f)\xi^{1}(f')} 
 \\
   &  &   + \sum_{(f,e) \in C} b_{p\xi^{1}(f)e, q\xi^{0}(f)\xi^{1}e}
   \end{eqnarray*}
 For fixed $q$ in $Q^{i}_{l}$, we have seen earlier in the proof 
 of Theorem \ref{groupoids:50} that 
 $b_{q,q} d^{1,0}_{p\xi^{1}(f), p'\xi^{0}(f)} $ is equal to 
 $d^{1,0}_{p\xi^{1}(f), p'\xi^{0}(f)}$ if $q =p\xi^{1}(f)$ and is zero
 otherwise. As $p \xi^{1}(f)$ is in $Q^{1}_{l}$ and not in $Q^{0}_{l}$, 
 summing over all $q$ and $i$ gives 
 \[
 \bar{e}_{l} d^{1,0}_{p\xi^{1}(f), p'\xi^{0}(f)} =  d^{1,0}_{p\xi^{1}(f), p'\xi^{0}(f)}.
 \]
 A similar argument, using the results of 
 Lemma \ref{groupoids:10}, shows the same conclusion for the second and third sum.
 By taking adjoints, we have the same conclusion for right multiplication by
 $\bar{e}_{l}$. That is, we have shown that 
 $ ( 1 - \bar{e}_{l})  d^{i, 1-i}_{p,p'} =  d^{i, 1-i}_{p,p'}  ( 1 - \bar{e}_{l}) =0$.
 
 Now, we assume that $l > n+1$. Again, we use the same decomposition of $d^{i,1-i}_{p,p'}$
 as above. In the first sum, we use the obvious fact that $(p\xi^{1}(f), p'\xi^{0}(f))$
 is in $I_{n+1}^{W}$ so that we may apply the induction hypothesis. 
 The conclusion follows at once for the first sum. 
 
 For the second sum, 
 the path $p\xi^{1}(f)\xi^{0}(f')$ is in neither $Q^{1}_{n+2}$ nor $Q^{0}_{n+2}$ 
 and it follows that
  $b_{q,q}  b_{p\xi^{1}(f)\xi^{0}(f'), q\xi^{0}(f)\xi^{1}(f')} =0$, for 
 all $q$ in $Q^{j}_{l}, j=0,1$. Similarly, since $e$ is in neither $\xi^{0}(F_{n+2})$
 nor $\xi^{1}(F_{n+2})$, the path 
 $p\xi^{1}(f)e$ is in neither $Q^{1}_{n+2}$ nor $Q^{0}_{n+2}$ 
 and it follows that $b_{q,q}  b_{p\xi^{1}(f)e, q\xi^{0}(f)e} =0$, for 
 all $q$ in $Q^{j}_{l}, j=0,1$. We have shown that left multiplication 
 by $\bar{e}_{l}$ on the second and third sum yields zero. Hence, multiplication
 by $1 - \bar{e}_{l}$ leaves these unchanged and they are already 
 in $B_{n+2} \subseteq B_{l}$.
\end{proof}

Now we can verify the three properties to show that $C^{*}(R)$ is tracially $AF$.
Our finite-dimensional $C^{*}$-subalgebra is $(1-\bar{e}_{l}) B_{l}(1 - \bar{e}_{l})$. 
We have already shown the first of the three properties, namely that 
$\Vert a (1 - \bar{e}_{l}) - (1-\bar{e}_{l}) a \Vert = 0 < \epsilon  $, for every $a$ in $\mathcal{F}$.

For the second property, we know that 
$(1 - \bar{e}_{l}) b_{p,q} (1 - \bar{e}_{l})$ lies in 
$(1 - \bar{e}_{l}) B_{n} (1 - \bar{e}_{l}) \subseteq (1 - \bar{e}_{l}) B_{l} (1 - \bar{e}_{l})$ 
For $(p,q)$ in $I_{n}$ with $n < l$. On the other hand, the fact that 
$(1 - \bar{e}_{l}) d^{i,1-i}_{p,q} (1 - \bar{e}_{l})$ lies in $(1 - \bar{e}_{l}) B_{l} (1 - \bar{e}_{l})$,
for $(p,q)$ in $I_{n}^{W}$, $n < l$, follows immediately from 
Lemma \ref{TAF:10}.

We finally turn to the proof that $l > n_{0}$ may be 
chosen so that $\bar{e}_{l}$ is unitarily equivalent to a subprojection
of $p_{0}$. We will make use of the following general facts which are 
familiar \cite{Bla}.
\begin{enumerate}
\item  If $a$ is a self-adjoint element of a unital  $C^{*}$-algebra
and $\Vert a^{2} -a \Vert < \epsilon < .25$ then there is a projection
$p$ with $\Vert a - p \Vert < \epsilon$. Or in words, 
almost being  a projection implies being close to a projection.
\item If $p$ and $q$ are two projections in a unital $C^{*}$-algebra 
with $\Vert p -q \Vert < 1$, then $p$ and $q$ 
are unitarily equivalent.
\end{enumerate}
Also recall that we have $a$ in the span of $\mathcal{F}$, and therefore
commutes with $\bar{e}_{l}$, and has $\Vert a - p  \Vert < \epsilon$. 
Moreover, $(1 - \bar{e}_{l}) a (1 - \bar{e}_{l}) $
lies in  the finite dimensional 
algebra $(1 - \bar{e}_{l}) B_{l} (1 - \bar{e}_{l}) $.

We will argue  heuristically, rather than making
 precise estimates which are  routine, although this does make some 
 assumptions on the size of 
 $\epsilon$. First, we claim that
 $(1-\bar{e}_{l}) p_{0}(1-\bar{e}_{l})$ is almost a projection which lies in 
 $C^{*}(R_{E})$.
 This follows from the fact that it is close to 
 $(1-\bar{e}_{l}) a(1-\bar{e}_{l}) $, which is almost a projection and lies in 
 $C^{*}(R_{E})$. It follows that it is close to a projection, 
 which we call $p_{1}$, 
 in $C^{*}(R_{E})$.
 
 Secondly, we claim that if $l$ is sufficiently large, then
 this projection is 
 not zero. Let $\pi$ be the representation of $C^{*}(R)$ 
 in Lemma \ref{TAF:5}. Since $p_{0}$ is assumed to be non-zero,
 it follows that 
  $\pi((1-\bar{e}_{l}) p_{0})$ is converging strongly to 
  $\pi(p_{0})$ which is non-zero since $\pi$ is faithful.
  It follows that, if $l$ is sufficiently large, 
$  (1-\bar{e}_{l}) p_{0} (1-\bar{e}_{l})
   =( (1-\bar{e}_{l}) p_{0}) ( (1-\bar{e}_{l}) p_{0})^{*}  $
   is not zero.
   
 Thirdly, we claim that 
 $(1-\bar{e}_{l}) p_{0}(1-\bar{e}_{l})$ is close to 
 $p_{0}(1-\bar{e}_{l}) p_{0}$. 
 This follows from the facts that $p_{0}$ is 
 close to $a$ which commutes with $\bar{e}_{l}$. 
 Finally,  $p_{0}(1-\bar{e}_{l}) p_{0}$ is almost a projection
 by similar arguments and is therefore close to a projection, 
 which we call $p_{2}$ in $p_{0}C^{*}(R)p_{0}$. Also note 
 that $p_{2}$ is close to  $p_{0}(1-\bar{e}_{l}) p_{0}$, 
 which is close to  $(1-\bar{e}_{l}) p_{0}(1-\bar{e}_{l})$, 
 which is close to $p_{1}$. 
 Hence $p_{2}$ and $p_{1}$ are unitarily equivalent.

 In summary, we have shown that $p_{0} \geq p_{2} \sim_{u} p_{1} \neq 0$
 and $p_{1}$ is in $C^{*}(R_{E})$. In a simple AF-algebra, positivity in 
 $K$-theory is determined entirely by traces \cite{Eff}. 
 As $\min_{\tau}\{ \tau(p_{1}) \} > 0$, where the minimum is taken over 
 all traces $\tau$ on $C^{*}(R_{E})$, we can find 
 $l$ sufficiently large such that 
 \[
 \tau(\bar{e}_{l}) \leq  2^{1+n_{0}-l} < \tau(p_{1}), 
 \]
 for every $\tau$. It follows that $ \bar{e}_{l}$ is unitarily 
 equivalent to a subprojection of $p_{1}$, and hence of $p_{0}$, as well.
 This completes the proof.

\bibliographystyle{amsplain}

\end{document}